\documentclass[final]{siamltex}

\usepackage{a4wide}
\usepackage{amssymb,amsmath, amsfonts}

\usepackage{color}

\newcommand{\red}[1]{\textcolor{red}{#1}}
\newcommand{\blue}[1]{\textcolor{blue}{#1}}

\newcommand{\at}[1]{}
\newcommand{\email}[1]{{\tt #1}}

\newcommand{\Gr}{{\rm gph\,}}

\newcommand{\dom}{{\rm dom\,}}

\newcommand{\ri}{{\rm ri\,}}

\newcommand{\xb}{\bar x}
\newcommand{\yb}{\bar y}

\newcommand{\pb}{\bar p}
\newcommand{\yba}{\yb^\ast}

\newcommand{\M}{{\cal M}}

\newcommand{\Kb}{\bar K}

\newcommand{\F}{{\mathcal F}}

\newcommand{\R}{\mathbb{R}}
\newcommand{\norm}[1]{\|#1\|}

\newcommand{\dist}[1]{{\rm d}(#1)}
\newcommand{\distSp}[2]{{\rm d}_{#1}(#2)}
\newcommand{\B}{{\cal B}}

\newcommand{\Sp}{{\cal S}}

\newcommand{\mv}{\,\vert\, }

\newcommand{\oo}{o}
\newcommand{\AT}[2]{{\textstyle{#1\atop#2}}}
\newcommand{\skalp}[1]{\langle #1\rangle}

\newcommand{\Limsup}{\mathop{{\rm Lim}\,{\rm sup}}}

\newtheorem{example}{Example}
\newtheorem{remark}{Remark}

\newcommand{\titlerunning}[1]{}
\newcommand{\institute}[1]{}

\begin{document}

\title{On Lipschitzian properties of implicit multifunctions}
\author{Helmut Gfrerer\footnotemark[2]
\and  Ji\v{r}\'{i} V. Outrata\footnotemark[3]
}
\date{}
\maketitle
\renewcommand{\thefootnote}{\fnsymbol{footnote}}
\footnotetext[2]{Institute of Computational Mathematics, Johannes Kepler University Linz,
              A-4040 Linz, Austria,
             \email{helmut.gfrerer@jku.at}}
\footnotetext[3]{Institute of Information Theory and Automation, Academy of Sciences of the Czech Republic, 18208 Prague, Czech Republic, and Centre for
              Informatics and Applied Optimization, Federation University of Australia, POB 663, Ballarat,  Vic 3350, Australia,  \email{outrata@utia.cas.cz}}
\renewcommand{\thefootnote}{\arabic{footnote}}

\begin{abstract}    The paper is devoted to the development of new sufficient conditions
for the calmness and the Aubin property of implicit
multifunctions. As the basic tool one employs the directional limiting
coderivative which, together with the graphical derivative, enable us a
fine analysis of the local behavior of the investigated multifunction
along relevant directions. For verification of the calmness property, in
addition, a new condition has been discovered which parallels the
missing implicit function paradigm and permits us to replace the
original multifunction by a substantially simpler one. Moreover, as an
auxiliary tool, a handy formula for the computation of the directional
limiting coderivative of the normal-cone map with a polyhedral set has
been derived which perfectly matches the framework of \cite{DoRo96}. All
important statements are illustrated by examples.
\end{abstract}

\pagestyle{myheadings}
\thispagestyle{plain}
\markboth{H. GFRERER AND J. V. OUTRATA}{LIPSCHITZIAN PROPERTIES OF IMPLICIT MULTIFUNCTIONS}

\section{Introduction}
Given a multifunction $M$ of two variables, say $p,x$, define the associated {\em implicit multifunction} $S$ by
\begin{equation}\label{eq-111}
S(p) :=\{x \mv 0 \in M(p,x)\}.
\end{equation}
The aim of this paper is to derive new conditions ensuring the {\em calmness} and the {\em Aubin (Lipschitz-like)} property of $S$ at or around the given reference point $(\pb,\xb)$, respectively. The definitions of these stability properties, together with several other notions, important for this development, are collected in Section 2.1.
Starting with the principal work of Dini \cite{Dini}, there is a large number of works dealing with the classical variant of (\ref{eq-111}), where $M$ is a (mostly smooth) single-valued map. The rapid development of modern variational analysis having started in the seventies has enabled, however, a step by step weakening of the assumptions imposed on $M$  and lead eventually to general multifunctional formulation (\ref{eq-111}). This modern framework has a lot of advantages and allows us to capture, for instance, various types of parameter-dependent constraint and variational systems, cf. \cite[Sections 4.3 and 4.4]{Mo06a}. From the long list of relevant references let us  mention the papers \cite{Ro1}, \cite{Ro2}, \cite{DH},  \cite{DoRo96}, \cite{A}, \cite{K3}, \cite{Thera} where the authors consider various special (frequently arising) classes of multifunctions $M$ and derive conditions  ensuring a Lipschitzian behavior of $S$. The recent monograph \cite{DoRo14} contains then a comprehensive presentation of  currently available results accompanied with a detailed explanation of   the so-called {\em implicit function paradigm}.
This paradigm facilitates substantially the derivation of conditions ensuring various stability properties of $S$ but, as pointed out in \cite[page 200]{DoRo14}, it does not work in the case of calmness. To overcome this hurdle, we have significantly improved an approach from \cite[Lemma 1]{HO}, concerning parameterized constraint systems. Our result (Theorem \ref{ThGenClm}) works  for general multifunctions $M$  and enables us to ensure  the calmness of $S$ via the {\em metric subregularity} of the mapping
\[M_{\pb}(x):=M(\pb,x)\]
at $(\xb,0)$ and a special relaxed calmness condition imposed on a mapping associated with $M$.
Our result has the same structure as \cite[Proposition 2.3]{DoQuZl06} where the authors state a criterion for the Aubin property of $S$ around $(\pb,\xb)$.

To ensure the metric subregularity of $M_ {\pb}$ one can employ one of the various approaches developed in the literature, see, e.g., \cite{IO, KF, K1, K2, Kru15a}. However, the derivative-like objects used in these papers do not possess a decent calculus and so it is difficult to compute them, e.g., in the case of parameterized constraint or variational systems.

The approach used in this paper is related to the techniques from \cite{Gfr11,Gfr13b}. It is based on the notion  of the {\em directional limiting coderivative} introduced in \cite{Gfr13a} (for a slightly different version see \cite{GiMo11}) which provides us with a convenient description of the local behavior of considered multifunctions {\em along} specified directions. Moreover, the directional coderivatives do possess a considerable calculus. The usage of this tool enables us not only to prove the calmness of $S$ (which was our main intention) but, in some cases, to ensure at the same time the non-emptiness of the sets $S(p)$ for $p$ close to $\pb$.

In contrast to the property of calmness  there already exists an efficient characterization of the Aubin property of $S$ in terms of a  derivative-like object associated with $M$. Herewith we mean the {\em Mordukhovich criterion} expressed via the limiting  coderivative, cf. \cite{Mo92}, \cite[Theorem 9.40]{RoWe98} and \cite{Kr82} for a preceding result of this sort.
Further efficient characterizations  can be found, e.g., in  \cite[Chapter 4.2]{DoRo14}.
Nevertheless, in some situations we are not able to compute the limiting coderivative of the implicitly given mapping $S$ precisely, and then resulting sufficient conditions can be far from necessity. This difficulty arises, e.g., when
\[
M(p,x) = G(p,x)+ Q(x),
\]
where $G$ is a continuously differentiable function with a nonsurjective partial Jacobian $\nabla_{p}G$ at the reference point $(\pb,\xb)$. We compute then only an upper estimate of the limiting coderivative of $S$, which makes the resulting condition too rough. Being motivated by this type of problems, we have again employed the directional limiting coderivative to construct a new, substantially weaker (less restrictive) criterion which is able to detect the Aubin property even if the existing  criteria based on the standard limiting coderivative fail. In such cases it suffices, namely, to examine the (local) behavior of $M$ only with respect to directions  for which the graphical derivative of $M$ at $(\pb,\xb,0)$ vanishes.

Both investigated properties, namely the calmness  and the Aubin property of $S$, belong to the basic stability properties of multifunctions. The generalized implicit multifunction model (\ref{eq-111}) is amenable  for a large class of parametric models ranging from constraint systems over variational inequalities up to complicated optimization and equilibrium problems. For all these problems the obtained conditions can be used as an efficient tool of post-optimal analysis.

The plan of the paper is as follows. The next ``preliminary'' section  is divided into three parts. The first one contains the basic definitions, whereas  the second one is devoted to the metric subregularity and its relationship with other notions like the directional metric (sub)regularity and the limiting directional coderivative. In the third part  we present several auxiliary results which are extensively used in the sequel. Some of them are interesting for their own sake and could be used also in a different context. In particular, in Theorem \ref{ThPolyDirLimNormCone} we present an easy-to-apply formula for the directional limiting normal cone to the graph of the normal-cone map associated with a convex polyhedron.  Sections 3 and 4 containing our main results are then devoted to the new criteria of the calmness and the Aubin property of $S$, respectively. The obtained results are illustrated by examples.

Our notation is basically standard.   In an Euclidean space, $\norm{\cdot}$ is the (Euclidean) norm and $\dist{x,\Omega}$ denotes the distance from a point $x$ to the set $\Omega$. Further, $\B_{\R^n}$ and $\Sp_{\R^n}$ denote the closed unit ball and the unit sphere in $\R^n$, respectively, and $\B(x,r):=\{u\mv \norm{u-x}\leq r\}$.
Given a metric space $X$, $\rho_X(\cdot,\cdot)$ stands for the corresponding metric, ${\rm dist}_X$ denotes the respective point-to-set distance function and $\B_X(x,r):=\{u\in X\mv \rho_X(u,x)\leq r\}$. Given the product $X\times Y$ of two (metric, Euclidean) spaces, we use the ``Euclidean'' metric
\[\rho_{X\times Y}((x,y),(x',y')):=\sqrt{(\rho_X(x,x'))^2+(\rho_Y(y,y'))^2}.\]
For a multifunction $F, \Gr F := \{(x,y)|y \in F(x)\}$ is its graph, ${\rm dom\;}F:=\{x\mv F(x)\not=\emptyset\}$ stands for its domain, ${\rm rge\;} F: = \{y \mv \exists x \mbox{ with } y \in F(x)\}$ denotes its range and $F^{-1}$ means the respective inverse mapping. Finally, $K^\circ$ is the (negative) polar cone to a cone $K$ and the notation of the objects from variational analysis together with the respective definitions  is introduced in the next section.

\section{\label{SecPrelim}Preliminaries}

\subsection{Basic notions}

Consider general closed-graph multifunctions $\M:X\rightrightarrows Y$ and $F: Y \rightrightarrows X$, where $X,Y$ are metric spaces.
\begin{definition}\label{DefAubinMetrReg}
{\bf (i)} We say that $F$ has the {\em Aubin property} around $(\yb,\xb)\in\Gr F$, provided there are reals $\kappa \geq 0$ and $\varepsilon > 0$ such that
\[ \distSp{X}{x,F(y)}\leq \kappa \rho_Y(v,y)\quad \mbox{ provided }\quad  \rho_Y(y,\yb)< \varepsilon, \rho_X(x,\xb)<\varepsilon, x \in F(v).
\]

{\bf (ii)}
$\M$   is said to be {\em metrically regular} around $(\xb,\yb)\in\Gr \M$, provided there are $\kappa \geq 0$ and $\varepsilon > 0$ such that
\[
\distSp{X}{x,\M^{-1}(y)} \leq \kappa \distSp{Y}{y,\M(x)}\quad \mbox{ provided }\quad  \rho_X(x,\xb)<\varepsilon, \rho_Y(y,\yb)<\varepsilon.
\]
\end{definition}
It is easy to see that $F$ has the Aubin property around $(\yb,\xb)$ if and only if $F^{-1}$ is metrically regular around $(\xb,\yb)$. The Aubin property has been introduced in \cite{Au 84} (under a different name) and since that time it has been widely used in  variational analysis both as a desired local stability property as well as in various qualification conditions in the nonsmooth calculus. It has also a close connection with the conclusions of theorems of Graves and Lyusternik.

\begin{definition}
In the setting of Definition \ref{DefAubinMetrReg} we say that\\

{\bf (i)}  $F$ is {\em calm} at $(\yb,\xb)$, provided there are reals $\kappa \geq 0$ and $\varepsilon > 0$ such that
\[
\distSp{X}{x,F(\yb)} \leq \kappa \rho_Y(y,\yb)\quad \mbox{ provided }\quad  \rho_Y(y,\yb)<\varepsilon, \rho_X(x,\xb)<\varepsilon\mbox{ and } x \in F(y).
\]

{\bf (ii)} $\M$ is  {\em metrically subregular} at  $(\xb,\yb)$, provided there are $\kappa \geq 0$ and $\varepsilon > 0$ such that
\[
\distSp{X}{x,\M^{-1}(\yb )} \leq \kappa \distSp{Y}{\yb,\M(x)}\quad \mbox{  provided }\quad  \rho_X(x,\xb)<\varepsilon.
\]
\end{definition}
Again, $F$ is calm at $(\yb,\xb)$ if and only if $F^{-1}$ is metrically subregular at $(\xb,\yb)$. Further we observe that the pair of properties from Definition 2 is strictly weaker (less restrictive) than their counterparts from Definition 1 and that the calmness of $F$ at $(\yb,\xb)$ does not entail the non-emptiness of $F(y)$ for $y$ close to $\yb$. As to our knowledge, the metric subregularity has been introduced in \cite{IO79} (under a different name) whereas  the calmness arose for the first time in the context of optimal value functions in \cite{Cl}. Later,  in \cite{Ye96} it has then be generalized to the form arising in Definition 2 (i) and used as a weak constraint qualification  (again under a different name). From the point of view of local post-optimal analysis it is, however, also a valuable property, in particular when one proves in addition that $F(y)\neq \emptyset$ on a neighborhood of $\yb$.

The above defined stability properties will be central in our development. To be able to conduct their thorough analysis in the investigated model, we will make use of several basic notions of the nonsmooth calculus stated below. Since we will be working with them only in finite dimensions, we will present their definitions below  in the finite-dimensional setting.

Let $A$ be a closed set in $\mathbb{R}^{s}$ and $\M$ be now a closed-graph multifunction mapping $\mathbb{R}^{s}$ into (sets of) $\mathbb{R}^{d}$.
\begin{definition}
Assume that $\xb\in A$. Then

{\bf (i)}
\[
T_{A}(\bar{x}):= \Limsup_{t \searrow 0} \frac{A-\bar{x}}{t}
\]
is the {\em tangent (contingent) cone} to $A$ at $\bar{x}$;

{\bf (ii)}
\[
\hat{N}_{A}(\bar{x}):=(T_{A}(\bar{x}))^{\circ}
\]
is the {\em regular  normal cone} to $A$ at $\bar{x}$;

{\bf (iii)}
\[
N_{A}(\bar{x}):= \Limsup_{x \stackrel{A}{\rightarrow}\bar{x}} \hat{N}_{A}(x)
\]
is the {\em limiting  normal cone} to $A$ at $\bar{x}$ and,

{\bf (iv)}
given a direction $u \in \mathbb{R}^{s}$,
\[
N_{A}(\xb;u):=\Limsup\limits_{\scriptstyle t \searrow 0 \atop \scriptstyle u^{\prime}\rightarrow u}
\hat{N}_{A}(\bar{x}+t u^{\prime})
\]
is the {\em directional limiting normal cone} to $A$ at $\bar{x}$ in direction $u$.

\end{definition}

The symbol ``Limsup'' in (i), (iii) and (iv) stands for the outer (upper) set limit in the sense of Painlev\'{e}-Kuratowski, cf. \cite[Chapter 4B]{RoWe98}. If $A$ is convex, then both the regular and the limiting normal cones coincide with the normal cone in the sense of convex analysis. Therefore we will use in this case the notation $N_{A}$.

We say that a tangent $u\in T_A(\xb)$ is {\em derivable} if there exists a mapping $\xi:[0,\varepsilon]\to A$ such that $\xi(0)=\xb$ and $\xi(t)-(\xb+tu)=\oo(t)$, cf. \cite[Definition 6.1]{RoWe98}. This notion arises also in the definition of the tangent cone in \cite{DuMi65}.

The above listed cones enable us to describe the local behavior of multifunctions via various generalized derivatives.

\begin{definition}
Consider a point $(\xb,\yb) \in \Gr \M$. Then

{\bf (i)}
the multifunction $D\M(\xb,\yb):\mathbb{R}^{s} \rightrightarrows\mathbb{R}^{d}$, defined by
\[
D\M(\xb,\yb)(u):= \{v \in \mathbb{R}^{d}| (u,v)\in T_{\Gr \M}(\xb,\yb)\}, u \in \mathbb{R}^{s}
\]
is called the {\em graphical derivative} of $\M$ at $(\xb,\yb)$;

{\bf (ii)} the multifunction $\hat{D}^\ast\M(\xb,\yb): \mathbb{R}^{d}\rightrightarrows\mathbb{R}^{s}$, defined by
\[
\hat{D}^\ast\M(\xb,\yb)(y^\ast):=\{x^\ast\in \mathbb{R}^{s} | (x^\ast,- y^\ast)\in \hat{N}_{\Gr \M}(\xb,\yb)\}, y^\ast\in \mathbb{R}^{d}
\]
is called the {\em regular  coderivative} of $\M$ at $(\xb,\yb)$.

{\bf (iii)} the multifunction $D^\ast \M(\xb ,\yb ): \mathbb{R}^{d}\rightrightarrows\mathbb{R}^{s}$, defined by
\[
D^\ast\M(\xb ,\yb )(y^\ast):=\{x^\ast\in \mathbb{R}^{s} | (x^\ast,- y^\ast)\in N_{\Gr \M}(\xb ,\yb )\}, y^\ast\in \mathbb{R}^{d}
\]
is called the {\em limiting  coderivative} of $\M$ at $(\xb ,\yb )$.

{\bf (iv)} Finally, given a pair of directions $(u,v) \in \mathbb{R}^{s} \times \mathbb{R}^{d}$, the multifunction $D^\ast\M((\xb , \yb ); (u,v)): \mathbb{R}^{d}\rightrightarrows\mathbb{R}^{s}$, defined by
\begin{equation}\label{eq-150}
D^\ast\M((\xb , \yb ); (u,v))(y^\ast):=\{x^\ast \in \mathbb{R}^{s} | (x^\ast,-y^\ast)\in N_{\Gr \M}((\xb , \yb ); (u,v)) \}, y^\ast\in \mathbb{R}^{d}
\end{equation}
is called the {\em directional limiting coderivative} of $\M$ at $(\xb , \yb )$ in direction $(u,v)$.
\end{definition}

For the properties of the cones (i)-(iii) from Definition 3 and generalized derivatives (i)-(iii) from Definition 4 we refer the interested reader to the monographs \cite{RoWe98} and \cite{Mo06a}. Various properties of the directional limiting normal cone and coderivative can be found in \cite{Gfr13a}, \cite{Gfr13b}, \cite{Gfr14a}, \cite{Gfr14b}, \cite{GfrKl15}. In the sequel we will make use of the fact that for a multifunction $\M:\R^s\rightrightarrows\R^d$ and a smooth mapping $h:\R^s\to\R^d$ one has (cf. \cite{Mo94})
\begin{eqnarray}
  \label{EqSumMF_Smooth_TangCone}T_{\Gr (h+\M)}(\xb,\yb)&=&\{(u,\nabla h(\xb)u+v))\mv (u,v)\in T_{\Gr \M}(\xb,\yb-h(\xb))\}\\
  \label{EqSumMF_Smooth_RegNormalCone}\widehat N_{\Gr (h+\M)}(\xb,\yb)&=&\{(x^\ast -\nabla h(\xb)^Ty^\ast,y^\ast)\mv (x^\ast,y^\ast)\in \widehat N_{\Gr \M}(\xb,\yb-h(\xb))\}
\end{eqnarray}
and consequently
\begin{equation}\label{EqSumMF_Smooth_LimCoDeriv}D^\ast(h+\M)((\xb , \yb ); (u,v))(y^\ast)=\nabla h(\xb)^Ty^\ast + D^\ast\M((\xb , \yb-h(\xb)); (u,v-\nabla h(\xb)u))(y^\ast).
\end{equation}

\subsection{Coderivative criteria for metric subregularity and calmness}

In this subsection we will summarize some conditions for metric subregularity given by the first author, which are used in the sequel. In addition, this subsection provides the reader with some geometrical insight essential for the results presented in the last section.

\if{
A weak sufficient condition for metric subregularity is provided by the following theorem.

\begin{theorem}
    \label{ThGenMS}Let $\M:\R^s\rightrightarrows\R^d$ be a multifunction with closed graph. Given $(\xb,\yb)\in\Gr \M$, assume that there do not exist sequences $t_k\searrow 0$, $u_k\in\Sp_{\R^s}$, $v_k\to 0_{\R^d}$,  $u_k^\ast\to 0_{\R^s}$, $v_k^\ast\in\Sp_{\R^s}$ with $v_k\not=0$,
  \begin{equation}\label{EqSOCalm1}
  u_k^\ast\in \hat D^\ast\M(x_k,y_k)(v_k^\ast)
  \end{equation}
  and
  \begin{equation}\label{EqSOCalm2}
  \lim_{k\to\infty}\frac{\skalp{v_k^\ast,y_k-\yb}}{\norm{y_k-\yb}}=1,
  \end{equation}
  where $(x_k,y_k):=(\xb,\yb)+t_k(u_k,v_k)$. Then $\M$ is metrically subregular at $(\xb,\yb)$.
\end{theorem}
\begin{proof}
  Follows from \cite[Corollary 1, Remark 1]{Gfr14a} with $u=0$.
\end{proof}

Note that the main difference between Theorem \ref{ThGenMS} and other sufficient conditions appearing in the literature (see, e.g., \cite{K1,K2,HenJouOut02, IO,LiMo12,KF}) is the inspection of the regular coderivative of $\M$ at points $(x_k,y_k)$ verifying $\lim_{k\to\infty}\norm{y_k-\yb}/\norm{x_k-\xb}=0$. For instance, in \cite{IO,KF} the authors consider the (regular) coderivative at points $(x_k,y_k)$ where $y_k$ belongs to the projection of $\yb$ onto $\M(x_k)$.

However, the assumptions of Theorem \ref{ThGenMS}, like those of other weak sufficient conditions for metric subregularity appearing in the literature, are very difficult to verify. The reason is that the property of metric subregularity is in general unstable under small perturbations, see e.g.\cite{DoRo14}, and this instability is reflected by the sufficient conditions. However, in applications it is important to have workable criteria and for this purpose we have to strengthen the assumptions of Theorem \ref{ThGenMS}.
}\fi
One can find numerous sufficient conditions for metric subregularity in the literature, see, e.g., \cite{K1,K2,Gfr14a,HenJouOut02, IO,Kru15a,LiMo12,KF}). However, these sufficient conditions are often very difficult to verify. The reason is that the property of metric subregularity is in general unstable under small perturbations, see e.g.\cite{DoRo14}, and this instability is reflected by the sufficient conditions. However, in applications it is important to have workable criteria and thus we are looking for some sufficient conditions for metric subregularity which are not as weak as possible but stable with respect to certain perturbations.

Consider the following definition.
\begin{definition}\label{DefLimSet}Let $\M:\R^s\rightrightarrows\R^d$ be a multifunction and let $(\xb,\yb)\in\Gr \M$.
  The {\em limit set critical for metric subregularity}, denoted by ${\rm Cr_0\,}\M(\xb,\yb)$, is the collection of all elements $(v,u^\ast)\in\R^d\times\R^s$ such that there are sequences $t_k\searrow 0$, $(u_k,v_k^\ast)\in\Sp_{\R^s}\times\Sp_{\R^d}$, $(v_k,u_k^\ast)\to(v,u^\ast)$ with $(-u_k^\ast,v_k^\ast)\in\widehat N_{\Gr\M}(\xb+t_ku_k,\yb+t_kv_k)$.
\end{definition}

\if{
Then in \cite[Theorem 3.2]{Gfr11} the following result was shown.
\begin{theorem}\label{ThStableC1MS}Let $\M:\R^s\rightrightarrows\R^d$ be a multifunction with closed graph and let $(\xb,\yb)\in\Gr \M$. If $(0,0)\not\in {\rm Cr_0\,}\M(\xb,\yb)$, then $\M$ is metrically subregular at $(\xb,\yb)$. Conversely, if $(0,0)\in {\rm Cr_0\,}\M(\xb,\yb)$, then there exists a continuously differentiable function $h:\R^s\to\R^d$ with $h(\xb)=0$ and $\nabla h(\xb)=0$ such that $\M+h$ is not metrically subregular at $(\xb,\yb)$.
\end{theorem}
}\fi

In \cite[Theorem 3.2]{Gfr11} it was shown that the condition $(0,0)\not\in {\rm Cr_0\,}\M(\xb,\yb)$ is sufficient for metric subregularity of $\M$ at $(\xb,\yb)$. We will now show that this criterion for metric subregularity  is stable under small $C^1$ perturbations. Moreover, we reformulate it in terms of directional limiting coderivatives to obtain the condition \eqref{EqEquivCrDirLimCoDer} below which is very congenial to the  Mordukhovich criterion for metric regularity, cf. \cite{Mo92}, \cite[Theorem 9.40]{RoWe98}.

\begin{remark}In \cite[Theorem 3.2]{Gfr11} an infinite-dimensional setting was considered and therefore a further limit set critical for metric subregularity denoted by ${\rm Cr\,}\M(\xb,\yb)$ appears. However, in finite dimensions both limit sets coincide: ${\rm Cr\,}\M(\xb,\yb)={\rm Cr_0\,}\M(\xb,\yb)$, cf. \cite[p.1450]{Gfr11}.
\end{remark}
\if{
\begin{remark}
  Note that the condition $(0,0)\not\in {\rm Cr_0\,}\M(\xb,\yb)$ differs from the assumptions of Theorem \ref{ThGenMS} by the absence of the requirements $v_k\not=0$ and \eqref{EqSOCalm2}.
\end{remark}

It is an easy consequence of  Theorem \ref{ThStableC1MS} that the condition $(0,0)\not\in {\rm Cr_0\,}\M(\xb,\yb)$ characterizes stability of metric subregularity under small $C^1$ perturbations.
}\fi
\begin{theorem}
  \label{ThStabMS}Let $\M:\R^s\rightrightarrows\R^d$ be a multifunction and let $(\xb,\yb)\in\Gr \M$. Then the following statements are equivalent:

  {\bf (i)} $(0,0)\not\in {\rm Cr_0\,}\M(\xb,\yb)$.

  {\bf (ii)} There exists $r>0$ such that for every $h\in C^1(\R^s,\R^d)$ with $h(\xb)=0$ and $\norm{\nabla h(\xb)}\leq r$ the mapping $\M+h$ is metrically subregular at $(\xb,\yb)$.

  {\bf (iii)}
  \begin{equation}\label{EqEquivCrDirLimCoDer}
 \forall 0\not=u\in\R^s:\ 0\in D^\ast\M((\xb,\yb);(u,0))(v^\ast)\Rightarrow  v^\ast=0.
\end{equation}
\end{theorem}
\begin{proof}
  By the second part of \cite[Theorem 3.2]{Gfr11} we have $\neg {\rm(i)}\Rightarrow \neg {\rm(ii)}$ and therefore ${\rm(ii)}\Rightarrow{\rm(i)}$ follows. We prove the reverse implication by contraposition. Assume that (ii) does not hold, i.e., there exists a sequence of functions $h_j\in C^1(\R^s,\R^d)$ with $h_j(\xb)=0$ and $\norm{\nabla h_j(\xb)}\leq 1/j$ such that $\M+h_j$ is not metrically subregular at $(\xb,\yb)$. By \cite[Theorem 3.2]{Gfr11},  for every $j$ we have $(0,0)\in{\rm Cr_0}(\M+h_j)(\xb,\yb)$ and hence for every $j$ there exist sequences $t_{jk}\searrow 0$, $(u_{jk},v_{jk}^\ast)\in\Sp_{\R^s}\times\Sp_{\R^d}$, $(v_{jk},u_{jk}^\ast)\to(0,0)$ with $(-u_{jk}^\ast,v_{jk}^\ast)\in\widehat N_{\Gr(\M+h)}(\xb+t_{jk}u_{jk},\yb+t_{jk}v_{jk})$. For every $j$ we can find some index $k(j)$ such that $t_{jk(j)}\leq 1/j$, $\norm{v_{jk(j)}}\leq 1/j$, $\norm{u_{jk(j)}^\ast}\leq 1/j$, $\norm{\nabla h_j(\xb+t_{jk(j)}u_{jk(j)})-\nabla h_j(\xb)}\leq 1/j$ and $\norm{h_j(\xb+t_{jk(j)}u_{jk(j)})-h_j(\xb)-t_{jk(j)}\nabla h_j(\xb)u_{jk(j)}}\leq t_{jk(j)}/j$. Putting $t_j:=t_{jk(j)}$, $u_j:=u_{jk(j)}$, $v_j:=v_{jk(j)}-h(\xb+t_ju_j)/t_j$, $v_j^\ast:=v_{jk(j)}^\ast$ and $u_j^\ast=u_{jk(j)}^\ast-\nabla h_j(\xb+t_ju_j)^Tv_j^\ast$ we obtain $(u_j,v_j^\ast)\in\Sp_{\R^s}\times\Sp_{\R^d}$, $\norm{u_j^\ast}\leq \norm{u_{jk(j)}^\ast}+\norm{\nabla h_j(\xb+t_ju_j)}\norm{v_j^\ast}\leq 3/j$ and
  \[\norm{v_j}\leq \norm{v_{jk(j)}}+\norm{h(\xb+t_ju_j)}/t_j\leq 1/j+ (t_j/j+\norm{h_j(\xb)+ t_j\nabla h_j(\xb)u_j})/t_j\leq 3/j.\]
  Since $(-u_{jk(j)}^\ast, v_j^\ast)\in \widehat N_{\Gr(\M+h)}(\xb+t_ju_j,\yb+t_jv_{jk(j)})$, by \eqref{EqSumMF_Smooth_RegNormalCone} we have
  \begin{eqnarray*}(-u_j^\ast,v_j^\ast)&=&(-u_{jk(j)}^\ast+\nabla h_j(\xb+t_ju_j)^Tv_j^\ast, v_j^\ast)\\
  &\in& \widehat N_{\Gr\M}(\xb+t_ju_j, \yb+t_jv_{jk(j)}-h_j(\xb+t_ju_j))=\widehat N_{\Gr\M}(\xb+t_ju_j,\yb+t_jv_j)\end{eqnarray*}
  and $(0,0)\in {\rm Cr_0\,}\M(\xb,\yb)$ follows. Hence ${\rm(i)}\Rightarrow{\rm(ii)}$ also holds and the equivalence between ${\rm (i)}$ and ${\rm (ii)}$ is established. The equivalence between ${\rm (i)}$ and ${\rm (iii)}$ follows from the definitions and the fact, that any sequence $(u_k,v_k^\ast)\in\Sp_{\R^s}\times\Sp_{\R^d}$ has a convergent subsequence.
\end{proof}

One easily concludes from Theorem \ref{ThStabMS}
that condition \eqref{EqEquivCrDirLimCoDer} implies the metric subregularity of $\M$ at $(\xb,\yb)$. In the sequel we will call this condition {\em first-order sufficient condition for metric subregularity} and use the acronym FOSCMS.

Conditions \eqref{EqEquivCrDirLimCoDer} examines the limiting coderivative only in directions of the form $(u,0)$, $u\not=0$ and therefore we have to look into normals to the graph of $\M$ at points $(x,y)$ with $\norm{y-\yb}=\oo(\norm{x-\xb})$.
\begin{remark}
  Note that condition \eqref{EqEquivCrDirLimCoDer} is in particular fulfilled, if either there is no direction $u\not=0$ with $0\in D\M(\xb,\yb)(u)$ or
  \[0\in D^\ast\M(\xb,\yb)(v^\ast)\Rightarrow  v^\ast=0.\]
  The former of these two special cases is equivalent to the so-called {\em strong metric subregularity}, cf. \cite[Theorem 4E.1]{DoRo14}, whereas the latter is equivalent to metric regularity of $\M$ by the Mordukhovich criterion, cf. \cite{Mo92},\cite[Theorem 9.40]{RoWe98}. However, the condition \eqref{EqEquivCrDirLimCoDer} is by far not restricted to these two special cases. A simple multifunction $\M$,
where condition \eqref{EqEquivCrDirLimCoDer} is fulfilled but the respective $S$ is neither strong metrically subregular nor
metrically regular, can be found in Example \ref{ExMetrSubreg}.
\end{remark}

To get more insight into the equivalences of Theorem \ref{ThStabMS}, consider the following definitions, cf.\cite{Gfr13a}.
\begin{definition}
Let $\M:\R^s\rightrightarrows\R^d$ be a multifunction and let $(\xb,\yb)\in\Gr\M$.

{\bf (i)} Given $(u,v)\in \R^s\times\R^d$, $\M$  is called {\em
 metrically regular in direction $(u,v)$} at
$(\xb,\yb)$, provided there exist  positive reals
$\delta>0$  and  $\kappa>0$ such that
\begin{equation}
\label{EqDirReg} \dist{\xb+tu',\M^{-1}(\yb+tv')}\leq \kappa \dist{\yb+tv',\M(\xb+tu')}
\end{equation}
holds for all $t\in[0,\delta]$ and all $(u',v')\in \B((u,v),\delta)$ with
$\dist{(\xb+tu',\yb+tv'),\Gr \M}\leq \delta t$.

{\bf (ii)} For given $u\in \R^s$, $\M$ is said to be
{\em metrically subregular in direction $u$} at $(\xb,\yb)$, if there are  positive reals $\delta>0$  and $\kappa'>0$
such that
\begin{equation}
\label{EqDirSubReg} \dist{\xb+tu',\M^{-1}(\yb)}\leq \kappa' \dist{\yb,\M(\xb+tu')}
\end{equation}
holds for all $t\in[0,\delta]$ and $u'\in \B(u,\delta)$.

The infimum of  $\kappa$ and $\kappa'$, respectively, over all such combinations of $\delta$, $\kappa$ and $\kappa'$, respectively, is called the modulus of the respective property.
\end{definition}

Note that these definitions imply that a multifunction $\M$ is automatically metrically regular in direction $(u,v)$ when  $(u,v)\not\in T_{\Gr\M}(\xb,\yb)$ and that $\M$ is metrically subregular in direction $u$ when  $(u,0)\not\in T_{\Gr\M}(\xb,\yb)$.

Metric subregularity in direction $u$ was introduced by Penot \cite{Pen98} (under the name  directional metric regularity). The above definition  of directional metric regularity is due to \cite{Gfr13a}. Note that in \cite{ArutAvIzm07,ArutIzm06} Arutyunov et al. have introduced and studied another notion of directional metric regularity which is an extension of an earlier notion used by \cite{BonSh00}.
\begin{lemma}
  \label{LemBasicPropDirMetr}
Let $\M:\R^s\rightrightarrows\R^d$ be a multifunction and let $(\xb,\yb)\in\Gr\M$.

{\bf (i)} Consider the following statements:\par
  \hskip\parindent(a) $\M$ is metrically regular  around $(\xb,\yb)$.\par
  \hskip\parindent(b) $\M$ is metrically regular  in direction $(0,0)$ at $(\xb,\yb)$.\par
  \hskip\parindent(c) $\M$ is metrically regular in every direction $(u,v)\not=(0,0)$.\par
Then $(a)\Leftrightarrow(b)\Rightarrow(c)$.\smallskip

{\bf (ii)} Consider the following statements:\par
  \hskip\parindent(a) $\M$ is metrically subregular  at $(\xb,\yb)$.\par
  \hskip\parindent(b) $\M$ is metrically subregular  in direction $(0,0)$ at $(\xb,\yb)$.\par
  \hskip\parindent(c) $\M$ is metrically subregular in every direction $u\not=0$.\par
Then $(a)\Leftrightarrow(b)\Leftrightarrow(c)$.\smallskip

{\bf (iii)} If $\M$ is metrically regular in direction $(u,0)$ at $(\xb,\yb)$, then it is also metrically subregular in direction $u$.
\end{lemma}
\begin{proof}
  Statement (i) follows immediately from the definition, statement (ii) follows from the definition and \cite[Lemma 2.7]{Gfr14b} and statement (iii) was shown in \cite[Lemma 1]{Gfr13a}.
\end{proof}

The following theorem is a directional extension of the Mordukhovich criterion \cite{Mo92},\cite[Theorem 9.40]{RoWe98} for metric regularity.
\begin{theorem}\label{ThExtMordCrit}
Let $\M:\R^s\rightrightarrows\R^d$ be a multifunction with closed graph and let $(\xb,\yb)\in\Gr\M$. Then $\M$ is metrically regular in direction $(u,v)\in\R^s\times\R^d$ at $(\xb,\yb)$ if and only if
\[ 0\in D^\ast\M((\xb,\yb);(u,v))(v^\ast)\Rightarrow  v^\ast=0.\]
\end{theorem}
Hence, condition \eqref{EqEquivCrDirLimCoDer} holds if and only if $\M$ is metrically regular in every direction $(u,0)$ with $u\not=0$. However, by Lemma \ref{LemBasicPropDirMetr} we see that for verifying metric subregularity of $\M$ we only need metric subregularity of $\M$ in every direction $u\not=0$. Assuming some special structure of the multifunction $\M$, directional metric subregularity can be ensured by a second-order sufficient condition. This is done  in the following theorem which is a specialized version of \cite[Theorem 4.3(2.)]{Gfr14b}.

\begin{theorem}\label{ThSecOrdDirMetrSub} Let $(\xb,\yb)$ belong to the graph of the mapping $\M(x)=G(x)+Q(x)$, where $G:\R^s\to\R^d$ is strictly differentiable at $\xb$ and $Q:\R^s\rightrightarrows \R^d$ is a polyhedral multifunction, i.e., its graph is the union of finitely many convex polyhedra. Further let $u\not=0$ and assume that the limit
\[G''(\xb;u):=\lim_{\AT{t\searrow 0}{u'\to u}}\frac {G(\xb+tu')-G(\xb)-t\nabla G(\xb)u'}{t^2/2}\]
exists. If the inequality
\[\skalp{v^\ast,G''(\xb;u)}<0\]
holds for every nonzero element $0\not=v^\ast\in\R^d$ satisfying
\[0\in D^\ast\M((\xb,\yb);(u,0))(v^\ast)=\nabla G(\xb)^Tv^\ast+D^\ast Q((\xb,\yb-G(\xb)); (u,-\nabla G(\xb)u))(v^\ast),\]
then $\M$ is metrically subregular in direction $u$ at $(\xb,\yb)$.
\end{theorem}

Note that the criterion of Theorem \ref{ThSecOrdDirMetrSub} is stable under perturbations $h\in C^2(\R^s,\R^d)$ with $h(\xb)=0$, $\nabla h(\xb)=0$ and $\norm{\nabla^2 h(\xb)}$ sufficiently small.

In the following corollary we summarize the preceding results for the special case of constraint systems, cf. also \cite[Corollary 1]{GfrKl15}.
\begin{corollary}\label{CorMS_ConstrSyst}
Let the multifunction $\M:\R^s\rightrightarrows R^d$ be  given by
$\M(x):=G(x)-D$, where $G:R^s\to\R^d$ is continuously
differentiable and $D\subset\R^d$ is a closed  set. Then $\M$
is metrically subregular at $(\xb,0)$ if one of the following
conditions is fulfilled:
\begin{enumerate}
\item {\em First order sufficient condition  for metric
subregularity (FOSCMS)}: For every $0\not=u\in\R^s$ with $\nabla
G(\xb)u\in T_D(G(\xb))$ one has
\[\nabla G(\xb)^Tv^\ast=0,\ v^\ast \in N_D(G(\xb);\nabla G(\xb)u)\ \Longrightarrow v^\ast=0.\]
\item {\em Second order sufficient condition for metric
subregularity (SOSCMS)}: $G$ is twice Fr\'echet differentiable at
$\xb$, $D$ is the union of finitely many convex polyhedra and
for every $0\not=u\in\R^s$ with $\nabla G(\xb)u\in
T_D(G(\xb))$ one has
\[\nabla G(x)^Tv^\ast=0,\ v^\ast\in N_D\big(G(\xb);\nabla G(\xb)u\big),
\ u^T\nabla^2 ({v^\ast}^T G)(\xb)u\geq0\ \Longrightarrow v^\ast=0.\]
\end{enumerate}
\end{corollary}
For a second order sufficient condition for metric subreguality of constraint systems $0\in G(x)-D$ when $D$ is convex but not necessarily polyhedral we refer to \cite{Gfr11}.

\if{At the end of this subsection we formulate a general weak calmness criterion which follows immediately  from Theorem \ref{ThGenMS}.
\begin{corollary}
  \label{CorGenClm}Let $F:\R^d\rightrightarrows\R^s$ be a multifunction with closed graph. Given $(\yb,\xb)\in\Gr F$, assume that there do not exist sequences $t_k\searrow 0$, $u_k\in\Sp_{\R^s}$, $v_k\to 0_{\R^d}$,  $u_k^\ast\to 0_{\R^s}$, $v_k^\ast\in\Sp_{\R^s}$ with $v_k\not=0$,
  \begin{equation*}
  -v_k^\ast\in \hat D^\ast F(y_k,x_k)(- u_k^\ast)
  \end{equation*}
  and
  \begin{equation*}
  \lim_{k\to\infty}\frac{\skalp{v_k^\ast,y_k-\yb}}{\norm{y_k-\yb}}=1,
  \end{equation*}
  where $(y_k,x_k):=(\yb,\xb)+t_k(v_k,u_k)$. Then $F$ is calm at $(\yb,\xb)$.
\end{corollary}
}\fi

\subsection{Auxiliary results}

In this section we consider the computation of the directional limiting normal cone to the graph of the normal cone mapping $N_\Gamma$ with a convex polyhedral set $\Gamma$. To this end, we introduce for each $(y,y^\ast)\in \Gr  N_\Gamma$ the {\em critical cone}
\[K_\Gamma(y,y^\ast):= T_\Gamma(y)\cap [y^\ast]^\perp,\]
where $[u]$ denotes the subspace $\{\alpha u\mv \alpha\in\R\}$ for any vector $u$. Further we set $K_\Gamma(y,y^\ast):=\emptyset$ if $(y,y^\ast)\not\in\Gr N_\Gamma$.

Now consider a fixed pair $(\yb,\yba)\in \Gr N_\Gamma$. Then it is well known, that the geometry of the normal cone mapping $N_\Gamma$ near $(\yb,\yba)$ coincides with the geometry of the normal cone mapping to $K_\Gamma(\yb,\yba)$ near $(0,0)$. In particular we have, cf. \cite[Lemma 2E.4]{DoRo14},
\begin{equation}\label{EqPolyRedLemma}
(\yb+v,\yba+v^\ast)\in\Gr N_\Gamma\ \Leftrightarrow\ (v,v^\ast)\in \Gr N_{K_\Gamma(\yb,\yba)} \mbox{ for $(v,v^\ast)$ sufficiently near to $(0,0)$},
\end{equation}
 and therefore
\begin{equation}\label{EqPolyTanCone} T_{\Gr N_\Gamma}(\yb,\yba)=\Gr N_{K_\Gamma(\yb,\yba)}.
\end{equation}
Further it was shown in \cite[Proof of Theorem 2]{DoRo96} that
\begin{equation}\label{EqPolyRegNormalCone}\widehat N_{\Gr N_\Gamma}(\yb,\yba)=(K_\Gamma(\yb,\yba))^\circ\times K_\Gamma(\yb,\yba)
\end{equation}
and that $N_{\Gr N_\Gamma}(\yb,\yba)$ is the union of all product sets $K^\circ\times K$ associated with cones $K$ of the form $F_1-F_2$, where $F_1$ and $F_2$ are closed faces of the critical cone $K_\Gamma(\yb,\yba)$ satisfying $F_2\subset F_1$.

Thanks to the definition of a face of a convex set, see \cite[Chapter 18]{Ro70}, the closed faces $F$ of any polyhedral convex cone $K$ are the polyhedral convex cones of the form
\[F= K\cap [z^\ast]^\perp\ \mbox{for some $z^\ast\in K^\circ$.}\]
We will denote the collection of all closed faces of a polyhedral convex cone $K$ by $\F(K)$.

In the following proposition we state a similar description of the directional limiting normal cone in terms of selected faces of the critical cone $K_\Gamma(\yb,\yba)$.
\begin{theorem}\label{ThPolyDirLimNormCone}
  Let $\Gamma$ be a convex polyhedral set in $\R^m$, and let $(\yb,\yba)\in\Gr N_\Gamma$ and $(v,v^\ast)\in T_{\Gr N_\Gamma}(\yb,\yba)$ be given. Then
  $N_{\Gr N_\Gamma}((\yb,\yba); (v,v^\ast))$ is the union of all product sets $K^\circ\times K$ associated with cones $K$ of the form $F_1-F_2$, where $F_1$ and $F_2$ are closed faces of the critical cone $K_\Gamma(\yb,\yba)$ satisfying $v\in F_2\subset F_1\subset [v^\ast]^\perp$.
\end{theorem}

In order to prove this theorem we need two preparatory lemmas.

\begin{lemma}\label{LemPolyDescr}Let $\Gamma$ be a convex polyhedral set in $\R^m$, and let $(\yb,\yba)\in\Gr N_\Gamma$. Then there exists some radius $\rho>0$ such that for every $(y,y^\ast)\in\Gr N_\Gamma\cap \B((\yb,\yba),\rho)$ one has
\begin{equation}
  \label{EqCritCone}K_\Gamma(y,y^\ast)= (K_\Gamma(\yb,\yba)\cap [y^\ast-\yba]^\perp) +[y-\yb].
\end{equation}
\end{lemma}
\begin{proof}
  From \cite[Proof of Theorem 2]{DoRo96} we can distill that there exists some radius $\rho>0$ such that for every $(y,y^\ast)\in\Gr N_\Gamma\cap \B((\yb,\yba),\rho)$ one has
  \begin{eqnarray}
    \label{EqPolyAuxCritCone1}K_\Gamma(y,y^\ast)&=&(T_\Gamma(\yb)\cap[y^\ast]^\perp)+[y-\yb]\\
    \label{EqPolyAuxCritCone2}T_\Gamma(\yb)\cap[y^\ast]^\perp&\subset& K_\Gamma(\yb,\yba).
  \end{eqnarray}
  Now consider $(y,y^\ast)\in\Gr N_\Gamma\cap \B((\yb,\yba),\rho)$. We will show that  $T_\Gamma(\yb)\cap[y^\ast]^\perp=K_\Gamma(\yb,\yba)\cap [y^\ast-\yba]^\perp$. Fix any $w\in T_\Gamma(\yb)\cap[y^\ast]^\perp$. By \eqref{EqPolyAuxCritCone2} we have $w\in K_\Gamma(\yb,\yba)$ and thus $w\in[\yba]^\perp$. Since $w\in [y^\ast]^\perp$ we obtain
  \[0=\skalp{y^\ast,w}=\skalp{\yba+(y^\ast-\yba),w}=\skalp{y^\ast-\yba,w}\]
  and $w\in T_\Gamma(\yb)\cap [\yba]^\perp\cap [y^\ast-\yba]^\perp=K_\Gamma(\yb,\yba)\cap [y^\ast-\yba]^\perp$ follows. This shows $T_\Gamma(\yb)\cap[y^\ast]^\perp\subset K_\Gamma(\yb,\yba)\cap [y^\ast-\yba]^\perp$. On the other hand we always have $[y^\ast]^\perp\supset [\yba]^\perp\cap[y^\ast-\yba]^\perp$ yielding $T_\Gamma(\yb)\cap[y^\ast]^\perp \supset T_\Gamma(\yb)\cap [\yba]^\perp\cap [y^\ast-\yba]^\perp=K_\Gamma(\yb,\yba)\cap [y^\ast-\yba]^\perp$. Hence the claimed relation $T_\Gamma(\yb)\cap[y^\ast]^\perp=K_\Gamma(\yb,\yba)\cap [y^\ast-\yba]^\perp$ holds true and the statement of the lemma follows from \eqref{EqPolyAuxCritCone1}.
\end{proof}

\begin{lemma}\label{LemFace}
  Let $K\subset \R^d$ be a convex polyhedral cone and let $(v,v^\ast)\in\Gr N_K$. Then
  \begin{equation}\label{EqFace1}
  \F((K\cap[v^\ast]^\perp)+[v])=\{F+[v]\mv F\in\F(K),  v\in F \subset [v^\ast]^\perp\}.
  \end{equation}
  and
  \begin{equation}\label{EqFace2}
  \{\tilde F_1-\tilde F_2\mv \tilde F_1,\tilde F_2\in\F((K\cap[v^\ast]^\perp)+[v]), \tilde F_2\subset \tilde F_1\}=\{F_1-F_2\mv F_1,F_2\in\F(K),  v\in F_2\subset F_1 \subset [v^\ast]^\perp\}.
  \end{equation}
\end{lemma}
\begin{proof}Note that for two convex polyhedral cones $K_1,K_2$ their polar cones $K_1^\circ$, $K_2^\circ$ and their sum $K_1+K_2$  are also convex polyhedral by  \cite[Corollaries 19.2.2, 19.3.2]{Ro70} and therefore closed. This implies $(K_1\cap K_2)^\circ=K_1^\circ +K_2^\circ$ and $(K_1+K_2)^\circ=K_1^\circ\cap K_2^\circ$ by \cite[Corollary 16.4.2]{Ro70}. Hence $\tilde K^\circ=(K^\circ+[v^\ast])\cap[v]^\perp$, where $\tilde K:=(K\cap[v^\ast]^\perp)+[v]$.
  Let $\tilde F\in \F(\tilde K)$ and consider $z^\ast\in \tilde K^\circ$ with $\tilde F=\tilde K\cap [z^\ast]^\perp$. From $z^\ast\in \tilde K^\circ$ we conclude $z^\ast\in [v]^\perp$ and $\tilde F=K\cap[v^\ast]^\perp\cap[z^\ast]^\perp+[v]$ follows. Further $z^\ast=w^\ast +\alpha v^\ast$ with $w^\ast\in K^\circ$  and $\alpha\in\R$, implying $[z^\ast]^\perp\cap[v^\ast]^\perp=[w^\ast]^\perp\cap[v^\ast]^\perp$. Since $v^\ast\in N_K(v)=K^\circ\cap[v]^\perp$, we obtain $w^\ast+v^\ast\in K^\circ$, $v\in [w^\ast+v^\ast]^\perp$ and
  $K\cap[w^\ast+v^\ast]^\perp=K\cap[w^\ast]^\perp\cap[v^\ast]^\perp$. This shows that $\tilde F= F+[v]$, where $F=K\cap[v^\ast]^\perp\cap[z^\ast]=K\cap[w^\ast+v^\ast]^\perp$ is a face of $K$ verifying $v\in F\subset[v^\ast]^\perp$. Conversely, let $F\in\F(K)$ satisfying $v\in F\subset[v^\ast]^\perp$ and choose $w^\ast\in K^\circ$ with $F=K\cap[w^\ast]^\perp$. Then $w^\ast+v^\ast\in K^\circ$ and $v\in =[w^\ast]^\perp\cap [v^\ast]^\perp=[w^\ast+v^\ast]^\perp\cap[v^\ast]^\perp$. Hence $w^\ast+v^\ast\in (\tilde K^\circ+[v^\ast])\cap[v]^\perp \cap[v]^\perp=\tilde K^\circ$ and $F+[v]=F\cap[v^\ast]^\perp+[v]=K\cap[w^\ast+v^\ast]^\perp\cap[v^\ast]^\perp+[v]=((K\cap [v^\ast]^\perp)+[v])\cap[w^\ast+v^\ast]^\circ\in\F(\tilde K)$ follows.\\
  In order to prove \eqref{EqFace2}, consider $\tilde F_1,\tilde F_2\in \F((K\cap[v^\ast]^\perp)+[v])$ with $\tilde F_2\subset \tilde F_1$ and, according to \eqref{EqFace1}, some corresponding faces $F_1,F_2\in \F(K)$ with $v\in F_i \subset [v^\ast]^\perp$ such that $\tilde F_i=F_i+[v]$, $i=1,2$. Now consider an arbitrary element $f_2\in F_2$. Then, due to $\tilde F_2\subset \tilde F_1$, there are $f_1\in F_1$ and $\alpha\in\R$ such that $f_2=f_1+\alpha v$. Expressing $\alpha$ as the difference of two nonnegative numbers $\alpha_1$ and $\alpha_2$, we obtain
  $f_2+\alpha_2 v=f_1+\alpha_1 v\in F_2\cap F_1$. Hence for all reals $\beta$ we have $f_2+\beta v=f_2+\alpha_2v +(\beta-\alpha_2)v\in (F_2\cap F_1)+[v]$ showing $\tilde F_2\subset (F_2\cap F_1)+[v]$. Since we obviously have $\tilde F_2\supset (F_2\cap F_1)+[v]$, the equality $\tilde F_2= F_2\cap F_1+[v]$ holds. The intersection $F_2'=F_2\cap F_1$ of  the  closed faces $F_1, F_2$ of $K$ is again a closed face of $K$, the inclusion $v\in F_2'\subset F_1\subset[v^\ast]^\perp$ obviously holds and we obtain
  \[\tilde F_1-\tilde F_2 =(F_1+[v])-(F_2'+[v])=F_1-F_2'+[\tilde v]=(F_1+\R^+\{v\})-(F_2'+\R_+\{v\})=F_1-F_2'.\]
  On the other hand, for faces $F_1,F_2$ of $K$ with $v\in F_2\subset F_1\subset [v^\ast]^\perp$ we have $F_2+[v]\subset F_1+[v]$ and
  \[F_1-F_2=(F_1+\R^+\{v\})-(F_2+\R_+\{v\})=F_1-F_2+[\tilde v]=(F_1+[v])-(F_2+[v]).\]
  Relation \eqref{EqFace2} now follows from this relation together with \eqref{EqFace1}.
\end{proof}

\begin{proof}(of Theorem \ref{ThPolyDirLimNormCone}) Let $\Kb$ denote the critical cone $K_\Gamma(\yb,\yba)$. Note that the requirement $(v,v^\ast)\in T_{\Gr N_\Gamma}(\yb,\yba)$ is equivalent to $(v,v^\ast)\in \Gr N_{\Kb}$ by virtue of (\ref{EqPolyTanCone}). This means that $v\in \Kb$ and $v^\ast \in N_{\Kb}(v)=\Kb^\circ\cap[v]^\perp$. Since $\Gr N_\Gamma$ is the union of finitely many convex polyhedrons, we can apply \cite[Lemma 3.4]{Gfr13b} together with \eqref{EqPolyRegNormalCone} to obtain
\begin{equation*}\label{EqPolyDirLimNormalCone1}N_{\Gr N_\Gamma}((\yb,\yba);(v,v^\ast))=\bigcup_{\AT{t\in(0,\bar t]}{(w,w^\ast)\in\B((v,v^\ast),\delta)}}(K_\Gamma(\yb+tw,\yba+tw^\ast))^\circ\times K_\Gamma(\yb+tw,\yba+tw^\ast)\end{equation*}
for all $\delta,\bar t>0$ sufficiently small.  By virtue of \eqref{EqPolyRedLemma}, we can choose $\delta$ and $\bar t$  small enough such that for every $t\in(0,\bar t)$ and every $(w,w^\ast)\in \B((v,v^\ast),\delta)$ the condition $(\yb+tw,\yba+tw^\ast)\in\Gr N_\Gamma$ is equivalent to $(w,w^\ast)\in \Gr N_{\Kb}$. By taking into account that $K_\Gamma(\yb+tw,\yba+tw^\ast)=\emptyset$ when $(\yb+tw,\yba+tw^\ast)\not\in\Gr N_\Gamma$, we arrive at the more precise statement
 \begin{equation}\label{EqPolyDirLimNormalCone2}N_{\Gr N_\Gamma}((\yb,\yba);(v,v^\ast))=\bigcup_{\AT{t\in(0,\bar t)}{(w,w^\ast)\in\B((v,v^\ast),\delta)\cap\Gr N_{\Kb}}}(K_\Gamma(\yb+tw,\yba+tw^\ast))^\circ\times K_\Gamma(\yb+tw,\yba+tw^\ast).\end{equation}
Further, by  decreasing $\bar t$ if necessary, we can also assume that $(\yb+tw,\yba+tw^\ast)\subset \Gr N_\Gamma\cap \B((\yb,\yba),\rho)$ holds for all $t\in(0,\bar t)$ and all $(w,w^\ast)\in\B((v,v^\ast),\delta)\cap\Gr N_{\Kb}$, where $\rho$ is given by Lemma \ref{LemPolyDescr}, implying  $K_\Gamma(\yb+tw,\yba+tw^\ast))=(\Kb\cap[tw^\ast]^\perp)+[tw]=(\Kb\cap[w^\ast]^\perp)+[w]=K_\Gamma(\yb+\bar tw,\yba+\bar tw^\ast))$
and
 \begin{equation}\label{EqPolyDirLimNormalCone3}N_{\Gr N_\Gamma}((\yb,\yba);(v,v^\ast))=\bigcup_{(w,w^\ast)\in\B((v,v^\ast),\delta)\cap\Gr N_{\Kb}}(K_\Gamma(\yb+\bar tw,\yba+\bar tw^\ast))^\circ\times K_\Gamma(\yb+\bar tw,\yba+\bar tw^\ast).\end{equation}
 By the Critical Superface Lemma \cite[Lemma 4H.2]{DoRo14} we obtain
 \begin{eqnarray*}\lefteqn{\{K_\Gamma(\yb+\bar tw,\yba+\bar tw^\ast)\mv (w,w^\ast)\in\B((v,v^\ast),\delta)\cap\Gr N_{\Kb}\}}\\
 &\hspace{4em}&= \{F_1-F_2\mv F_1,F_2\in \F(K_\Gamma(\yb+\bar tv,\yba+\bar tv^\ast)),\ F_2\subset F_1\}
 \end{eqnarray*}
 for every $\delta>0$ sufficiently small and the statement of the theorem follows from \eqref{EqPolyDirLimNormalCone3} and Lemmas \ref{LemPolyDescr},\ref{LemFace}.
\end{proof}

\if{
\begin{lemma}\label{LemPolyAux1}Let $\Gamma$ be a convex polyhedral set in $\R^m$, and let $(\yb,\yba)\in\Gr N_\Gamma$. Then there exists some radius $\rho>0$ such that for every $(y,y^\ast)\in\Gr N_\Gamma\cap \B((\yb,\yba),\rho)$ the following properties hold:
\begin{enumerate}
  \item[(i)]$K_\Gamma(y,y^\ast)=(T_\Gamma(\yb)\cap[y^\ast]^\perp)+[y-\yb]$;
  \item[(ii)]$T_\Gamma(\yb)\cap[y^\ast]^\perp$ is a closed face of $K_\Gamma(\yb,\yba)$.
\end{enumerate}
\end{lemma}
\begin{proof}Follows from \cite[Proof of Theorem 2]{DoRo96}.
\end{proof}
\begin{lemma}\label{LemPolyAux2}Let $C_1,C_2\in\R^m$ be two  convex cones satisfying $C_1\supset C_2$. Then for every $v\in\ri C_2$ there holds
\[C_1-C_2 = C_1+[v]. \]
\end{lemma}
\begin{proof}
 Consider $c_1\in C_1$, $\lambda\in\R$ and choose $\alpha,\beta>0$ such that $\lambda=\alpha-\beta$. Then $c_1+\alpha v\in C_1$ and $\beta v\in C_2$ showing
 $c_1+\lambda v=(c_1+\alpha v)- \beta v\in C_1-C_2$ and we conclude $C_1+[v]\subset C_1-C_2$. To show the reverse inclusion, consider $c_1\in C_1$ and $c_2\in C_2$. Since $v\in\ri C_2$, we can find $\lambda>1$ such that $\lambda v+(1-\lambda)c_2=:\bar c_2\in C_2$, cf. \cite[Theorem 6.4]{Ro70}. Hence $c_2=\frac 1{\lambda -1}(\lambda v-\bar c_2)$ and, since $C_1$ is a convex cone containing $C_2$, we obtain
 \[c_1-c_2=(c_1+\frac 1{\lambda-1} \bar c_2)-\frac{\lambda}{\lambda-1} v\in C_1+[v].\]
 Thus $C_1-C_2\subset C_1+[v]$ also holds and this completes the proof.
\end{proof}
\begin{proof}(of Theorem \ref{ThPolyDirLimNormCone}) Let $\Kb$ denote the critical cone $K_\Gamma(\yb,\yba)$. Note that the requirement $(v,v^\ast)\in T_{\Gr N_\Gamma}(\yb,\yba)$ is equivalent to $(v,v^\ast)\in \Gr N_{\Kb}$ by virtue of (\ref{EqPolyTanCone}). This means that $v\in \Kb$ and $v^\ast \in N_{\Kb}(v)=\Kb^\circ\cap[v]^\perp$. Since $\Gr N_\Gamma$ is the union of finitely many convex polyhedrons, we can apply \cite[Lemma 3.4]{Gfr13b} together with \eqref{EqPolyRegNormalCone} to obtain
\begin{equation*}\label{EqPolyDirLimNormalCone1}N_{\Gr N_\Gamma}((\yb,\yba);(v,v^\ast))=\bigcup_{\AT{t\in(0,\delta)}{(w,w^\ast)\in\B((v,v^\ast),\delta)}}(K_\Gamma(\yb+tw,\yba+tw^\ast))^\circ\times K_\Gamma(\yb+tw,\yba+tw^\ast)\end{equation*}
for all $\delta>0$ sufficiently small.  By virtue of \eqref{EqPolyRedLemma}, we can choose $\delta$ small enough such that for every $t\in(0,\delta)$ and every $(w,w^\ast)\in \B((v,v^\ast),\delta)$ the condition $(\yb+tw,\yba+tw^\ast)\in\Gr N_\Gamma$ is equivalent to $(w,w^\ast)\in \Gr N_{\Kb}$. By taking into account that $K_\Gamma(\yb+tw,\yba+tw^\ast)=\emptyset$ when $(\yb+tw,\yba+tw^\ast)\not\in\Gr N_\Gamma$, we arrive at the more precise statement
 \begin{equation}\label{EqPolyDirLimNormalCone2}N_{\Gr N_\Gamma}((\yb,\yba);(v,v^\ast))=\bigcup_{\AT{t\in(0,\delta)}{(w,w^\ast)\in\B((v,v^\ast),\delta)\cap\Gr N_{\Kb}}}(K_\Gamma(\yb+tw,\yba+tw^\ast))^\circ\times K_\Gamma(\yb+tw,\yba+tw^\ast).\end{equation}
Further, by  decreasing $\delta$ if necessary, we can also assume that $(\yb+tw,\yba+tw^\ast)\subset \Gr N_\Gamma\cap \B((\yb,\yba),\rho)$ holds for all $t\in(0,\delta)$ and all $(w,w^\ast)\in\B((v,v^\ast),\delta)\cap\Gr N_{\Kb}$, where $\rho$ is given by Lemma \ref{LemPolyAux1}. Applying Lemma \ref{LemPolyAux1}(ii) once more with data $\Gamma$ and $(\yb,\yba)$ replaced by $\Kb$ and $(0,v^\ast)$, respectively, we can also assume that $T_{\Kb}(0)\cap [w^\ast]^\perp=\Kb\cap [w^\ast]^\perp$ is a face of $K_{\Kb}(0,v^\ast)=\Kb\cap [v^\ast]^\perp$ for every $(w',w^\ast)\in \Gr N_{\Kb}\cap \B((0,v^\ast),\delta)$. In particular, by taking $w'=0$ we obtain that
\begin{equation}\label{EqPolyAux1}\Kb\cap [w^\ast]^\perp\subset \Kb\cap [v^\ast]^\perp \; \forall w^\ast\in \Kb^\circ\cap \B(v^\ast,\delta).
\end{equation}
Since $\Kb$ is a convex polyhedral cone, we can choose $\delta>0$  small enough such that, in addition, one has $v+2(w-v)\in \Kb$ for all $w\in\B(v,\delta)\cap \Kb$.

In order to prove the \red{theorem} we have to show that the cones $K_\Gamma(\yb+tw,\yba+tw^\ast)$ appearing in \eqref{EqPolyDirLimNormalCone2} are precisely the cones of the form $F_1-F_2$, where $F_1$ and $F_2$ are closed faces of the critical cone $\Kb$ satisfying $v\in F_2\subset F_1\subset [v^\ast]^\perp$. We will proceed in a similar way as in \cite[Proof of Lemma 4H.2]{DoRo14}.

 So, consider any $t\in(0,\delta)$, any $(w,w^\ast)\in \B((v,v^\ast),\delta)\cap\Gr N_{\Kb}$ and set $F_1:=T_\Gamma(\yb)\cap[\yba+tw^\ast]^\perp$. By Lemma \ref{LemPolyAux1}(ii) $F_1$  is a closed face of $\Kb$ and, since $w^\ast\in N_{\Kb}(w)=\Kb\cap[w]^\perp\subset [w]^\perp$ and $w\in \Kb=T_\Gamma(\yb)\cap [\yba]^\perp$, we have $w\in F_1$. By taking $F_2$ as the closed face of $F_1$ having $w$ in its relative interior, from Lemma \ref{LemPolyAux2} together with Lemma \ref{LemPolyAux1}(i) we deduce that
\[K_\Gamma(\yb+tw,\yba+tw^\ast)=(T_\Gamma(\yb)\cap[\yba+tw^\ast]^\perp)+[tw]=(T_\Gamma(\yb)\cap[\yba+tw^\ast]^\perp)+[w]=F_1-F_2.\]
Since $F_2$ is a face of  $F_1$ which is itself a face of $\Kb$, we conclude that $F_2$ is a face of $\Kb$. Since $w\in F_2$ is in the relative interior of the line segment connecting $v,v+2(w-v)\in \Kb$, both endpoints $v$ and $v+2(w-v)$ must belong to $F_2$ (see \cite[Chapter 18]{Ro70}), in particular $v\in F_2$. Next we show that $F_1\subset [v^\ast]^\perp$. Since $F_1\subset \Kb\subset[\yba]^\perp$, for every $z\in F_1\subset[\yba +tw^\ast]^\perp$ we have $z\in [w^\ast]^\perp$ implying $F_1\subset \Kb\cap [w^\ast]^\perp$. On the other hand, $[\yba+tw^\ast]^\perp\supset [\yba]^\perp\cap[w^\ast]^\perp$ and thus $F_1\supset T_\Gamma(\yb)\cap [\yba]^\perp\cap[w^\ast]^\perp=\Kb\cap [w^\ast]^\perp$. Hence $F_1=\Kb\cap [w^\ast]^\perp$ and since $w^\ast\in N_{\Kb}(w)\cap \B(v^\ast,\delta)\subset \Kb^\circ\cap \B(v^\ast,\delta)$, we obtain $F_1\subset [v^\ast]^\perp$ from \eqref{EqPolyAux1}. Thus the desired representation of $K_\Gamma(\yb+tw,\yba+tw^\ast)$ is achieved.

Conversely, if $K:=F_1-F_2$ for closed faces $F_1$ and $F_2$ of $\Kb$ with $v\in F_2\subset F_1\subset [v^\ast]^\perp$, then we claim the existence of $(w,w^\ast)\in\B((v,v^\ast),\delta)\cap \Gr N_{\Kb}$ and $t\in (0,\delta)$ such that $K=K_\Gamma(\yb+tw,\yba+tw^\ast)$.
Since $F_1$ is also a closed face of $\Kb\cap [v^\ast]^\perp$, there exists some $z^\ast\in (\Kb\cap [v^\ast]^\perp)^\circ=\Kb^\circ+[v^\ast]$ with $F_1=\Kb\cap [v^\ast]^\perp\cap[z^\ast]^\perp$. Then $z^\ast=k^\ast+\lambda v^\ast$ with $k^\ast\in \Kb^\circ$ and $\lambda\in\R$. Let $z\in\ri F_2$ and choose
$\tau\in (0,1]$ so small that $1-\tau+\tau\lambda>0$ and $(w,w^\ast):=(1-\tau)(v,v^\ast)+\tau(z,z^\ast)\in \B((v,v^\ast),\delta)$. Since $v\in F_2$ we have $w\in\ri F_2$ and thus $K=F_1+[w]$ by Lemma \ref{LemPolyAux2}. Further we have $w^\ast=\tau k^\ast+(1-\tau+\tau\lambda)v^\ast\in \Kb^\circ$ and
$w\in F_2\subset F_1= \Kb\cap[v^\ast]^\perp\cap [z^\ast]^\perp\subset \Kb\cap[w^\ast]^\perp$ and consequently $(w,w^\ast)\in\Gr N_{\Kb}\cap \B((v,v^\ast),\delta)$. Since both $k^\ast$ and $v^\ast$ belong to $\Kb^\circ=(T_\Gamma(\yb)\cap[\yba]^\perp)^\circ=N_\Gamma(\yb)+[\yba]$, we can find elements $n_k^\ast, n_v^\ast\in N_\Gamma(\yb)$ and reals $\gamma_k,\gamma_n$ with  $k^\ast=n_k^\ast+\gamma_k \yba$ and $v^\ast=n_v^\ast+\gamma_v\yba$. Now choose $t\in(0,\delta)$ so that $1+t(\tau\gamma_k+(1-\tau+\tau\lambda)\gamma_v)>0$. For every $z'\in T_\Gamma(\yb)\cap[\yba+tw^\ast]^\perp$ we have
\begin{eqnarray*}0&=&\skalp{\yba+tw^\ast,z'}=\skalp{\yba+t(\tau k^\ast +(1-\tau+\tau\lambda)v^\ast),z'}\\ &=&(1+t(\tau\gamma_k+(1-\tau+\tau\lambda)\gamma_v))\skalp{\yba,z'}+
t\tau\skalp{n_k^\ast,z'}+t(1-\tau+\tau\lambda)\skalp{n_v^\ast,z'}.
\end{eqnarray*}
Since the elements  $\yba$, $n_k^\ast$ and $n_v^\ast$ belong to $N_\Gamma(\yb)=(T_\Gamma(\yb))^\circ$, we have $\skalp{\yba,z'}\leq 0$, $\skalp{n_k^\ast,z'}\leq 0$ and $\skalp{n_v^\ast,z'}\leq 0$. It follows that $\skalp{\yba,z'}=\skalp{n_k^\ast,z'}=\skalp{n_v^\ast,z'}=0$ and hence $\skalp{v^\ast,z'}=\skalp{z^\ast,z'}=0$, which implies that
 \[T_\Gamma(\yb)\cap[\yba+tw^\ast]^\perp\subset T_\Gamma(\yb)\cap[\yba]^\perp\cap[v^\ast]^\perp\cap[z^\ast]^\perp
 =\Kb\cap[v^\ast]^\perp\cap[z^\ast]^\perp=F_1.\]
   Conversely, by virtue of  $[\yba]^\perp\cap[v^\ast]^\perp\cap[z^\ast]^\perp\subset[\yba]^\perp\cap[w^\ast]^\perp\subset[\yba+tw^\ast]^\perp$ we obtain
\begin{eqnarray*}T_\Gamma(\yb)\cap[\yba]^\perp\cap[v^\ast]^\perp\cap[z^\ast]^\perp&\subset& T_\Gamma(\yb)\cap[\yba+tw^\ast]^\perp,
\end{eqnarray*}
so that $F_1= \Kb\cap[v^\ast]^\perp\cap[z^\ast]^\perp=T_\Gamma(\yb)\cap[\yba+tw^\ast]^\perp$. By using Lemma \ref{LemPolyAux1}(i) we obtain that \[F_1-F_2=(T_\Gamma(\yb)\cap[\yba+tw^\ast]^\perp)+[w]=(T_\Gamma(\yb)\cap[\yba+tw^\ast]^\perp)+[tw]=K_\Gamma(\yb+tw,\yba+tw^\ast)\]
has the required form.
\end{proof}
}\fi

\section{Calmness of implicit multifunctions}
This section is divided into two parts. In the first one we prove that the calmness of $F$ is ensured by the two conditions imposed on $M$, mentioned already in the Introduction. In contrast to the remainder of the paper,  Subsection \ref{SubSec3.1} is formulated in the setting of general metric spaces. Subsection \ref{SubSec3.2} is then focused on the question how these two conditions can be verified by using the tools of  variational analysis.

\subsection{\label{SubSec3.1}General theory}
Let $P,X,Y$ be metric spaces.
With respect to this general setting we will analyze now, instead of \eqref{eq-111}, the multifunction
\begin{equation}\label{eq-112}
S(p):= \{x \in X | \yb \in M(p,x)\},
\end{equation}
where $M:P\times X\rightrightarrows Y$ is a given multifunction and $\bar{y}$ is a given element of $Y$.

In \cite[Lemma 1]{HO} the authors considered the special case
\begin{equation}
  \label{EqConstrSyst}M(p,x)=G(p,x)-D,
\end{equation}
where the function $G:P\times X\to Y$ is Lipschitz near the reference pair $(\pb,\xb)$ and $D$ is a closed subset of $Y$. When $P,X,Y$ are normed spaces, it has been shown therein that the calmness of the respective multifunction $S$ at $(\pb,\xb)$ is implied by the metric subregularity of the (simpler) mapping $M_{\pb}:X\rightrightarrows Y$ defined by
\[M_{\pb}(x):=M(\pb,x)=G(\pb,x)-D\]
at $(\xb,0)$.

Next we present a deep generalization of this result which is valid even in our general setting in metric spaces and in which the structural assumption \eqref{EqConstrSyst} is abandoned.   We associate with $M$ the multifunctions $H_M:P\rightrightarrows X\times Y$ and $M_p:X\rightrightarrows Y$ defined by
\begin{equation}\label{EqH_M}
 H_M(p):=\{(x,y)\mv y\in M(p,x)\}
\end{equation}
and
\begin{equation}\label{EqM_p}
M_p(x):=\{y\mv y\in M(p,x)\}\ \mbox{ for each $p\in P$.}
\end{equation}
Note that $\Gr H_M=\Gr M$.
The following auxiliary notion will be crucial for our analysis.

\begin{definition}
  Let $(\pb,\xb,\yb)\in\Gr M$. We say that $M$ has  the {\em  restricted calmness property with respect to $p$  at $(\pb,\xb,\yb)$}, if there are reals $L\geq 0$  and $\epsilon>0$ such that
  \begin{equation}\label{EqRPClmP}
    \distSp{X\times Y}{(x,\yb),H_M(\pb)}\leq L\rho_P(p,\pb)\mbox{ provided  $\rho_P(p,\pb)<\varepsilon$, $\rho_X(x,\xb)<\varepsilon$ and $(x,\yb)\in H_M(p)$.}
  \end{equation}
\end{definition}

\begin{remark}
It is easy to see that $M$ has the restricted calmness property with respect to $p$  at $(\pb,\xb,\yb)$, if $H_M$ is calm at $(\pb,(\xb,\yb))$. In particular, in the setting of normed spaces this condition is fulfilled for multifunctions of the form $M(p,x)=G(p,x)+Q(x)$, where $G$ is a Lipschitz continuous function and $Q$ is set-valued.
\end{remark}

The following lemma states that the restricted calmness property with respect to $p$  is necessary for the calmness of the solution mapping $S$.

\begin{lemma}If $S$ is calm at $(\pb,\xb)$, then $M$ has the restricted calmness property with respect to $p$  at $(\pb,\xb,\yb)$.
\end{lemma}
\begin{proof}
  According to the definition of calmness we choose   reals $L\geq 0$ and $\epsilon>0$ such that
  \[\distSp{X}{x,S(\pb)}\leq L\rho_P(p,\pb)\mbox{ provided $p\in\B_P(\pb,\varepsilon)$ and $x\in S(p)\cap \B_X(\xb,\varepsilon)$}.\]
  Next, for every $p\in \B_P(\pb,\epsilon)$ and every $(x,\yb)\in H_M(p)$ with $x\in \B_X(\xb,\varepsilon)$ we have $x\in S(p)\cap \B_X(\xb,\varepsilon)$. Clearly, for each $\alpha>0$ there is some $\tilde x\in S(\pb)$ with $\rho_X(x,\tilde x)\leq \distSp{X}{x,S(\pb)}+\alpha$ and so it follows that $\rho_X(x,\tilde x)\leq L\rho_P(p,\pb)+\alpha$. Note that $(\tilde x,\yb)\in H_M(\pb)$, whence
  \[\distSp{X\times Y}{(x,\yb),H_M(\pb)}\leq \rho_X(x,\tilde x)\leq L\rho_P(p,\pb)+\alpha.\]
  Since a suitable point $\tilde x$ can be found for any arbitrary small $\alpha$, one can conclude that
  \[\distSp{X\times Y}{(x,\yb),H_M(\pb)}\leq L\rho_P(p,\pb),\]
  which amounts to the restricted calmness property with respect to $p$  of $M$ at $(\pb,\xb,\yb)$.
\end{proof}

We state now a sufficient criterion for the calmness of $S$.

\begin{theorem}\label{ThGenClm}
  Let $\yb\in M(\pb,\xb)$, assume that $M$ has the restricted calmness  property with respect to $p$  at $(\pb,\xb,\yb)$ and  $M_{\pb}$ is metrically subregular at $(\xb,\yb)$. Then $S$ is calm at $(\pb,\xb)$.
\end{theorem}
\begin{proof}
  By virtue of the restricted calmness property with respect to $p$  of $M$ at $(\bar{p},\bar{x},\bar{y}$)  we can find moduli $L$ and $\kappa$ along with some radii $r_p,r_x,\sigma>0$ such that
\begin{equation}
    \label{EqRPClmP1}
    \distSp{X\times Y}{(x,\yb),H_M(\pb)}\leq L\rho_P(p,\pb)\mbox{ provided  $p\in\B_P(\pb,r_p)$, $x\in\B_X(\xb,r_x)$ and $(x,\yb)\in H_M(p)$}
\end{equation}
and, by the metric subregularity of $M_{\bar{p}}$ at $(\bar{x}, \bar{y})$, one has
\begin{equation}
  \label{EqSubregRestrCalm}
  \distSp{X}{x,M_{\pb}^{-1}(\yb)}\leq\kappa \distSp{Y}{\yb,M_{\pb}(x)}\mbox{ provided $x\in \B_X(\xb,\sigma)$.}
\end{equation}
By decreasing the radii $r_p$ and $r_x$ if necessary we can assume $r_x+Lr_p<\sigma$. Now fix $p\in \B_P(\pb,r_p)$ and consider $x\in S(p)\cap \B_X(\xb,r_x)$ so that $(x,\yb)\in H_M(p)\cap(\B_X(\xb,r_x)\times\{\yb\})$. Further observe that for each $\beta>0$ there is some $(\tilde x,\tilde y)\in H_M(\pb)$ satisfying
\begin{equation}\label{EqAuxDist}
  \rho_{X\times Y}((x,\yb),(\tilde x,\tilde y))\leq \distSp{X\times Y}{(x,\yb), H_M(\pb)}+\beta.
\end{equation}
Consequently, by virtue of \eqref{EqRPClmP1}
\begin{equation}\label{EqAuxDist1}
  \rho_{X\times Y}((x,\yb),(\tilde x,\tilde y))\leq L\rho_P(p,\pb)+\beta.
\end{equation}
It follows from the triangle inequality and \eqref{EqAuxDist1} that
\[\rho_X(\tilde x,\xb)\leq\rho_X(\tilde x,x)+\rho_X(x,\xb)\leq L\rho_P(p,\pb)+\beta+\rho_X(x,\xb)\leq Lr_p+\beta+r_x.\]
Since $Lr_p+r_x<\sigma$, $\beta$ can be chosen sufficiently small to obtain that $\rho_X(\tilde x,\xb)<\sigma$ as well. Further we note that to each $\gamma>0$ there is some $\hat x\in S(\pb)=M_{\pb}^{-1}(\yb)$ such that
\[\rho_X(\tilde x,\hat x)\leq \distSp{X}{\tilde x,M_{\pb}^{-1}(\yb)}+\gamma.\]
This implies, by virtue of \eqref{EqSubregRestrCalm}, that
\begin{equation}
  \label{EqAuxDist2} \rho_X(\tilde x,\hat x)\leq \kappa \distSp{Y}{\yb,M_{\pb}(\tilde x)}+\gamma\leq \kappa\rho_Y(\yb,\tilde y)+\gamma,
\end{equation}
where the last inequality follows from the fact that $\tilde y\in M_{\pb}(\tilde x)$.

Hence, by using successively the triangle inequality, estimate \eqref{EqAuxDist2}, the Cauchy-Schwartz inequality and estimates \eqref{EqAuxDist1}, \eqref{EqRPClmP1}, we obtain
\begin{eqnarray*}
\distSp{X}{x,S(\pb)}&\leq& \rho_X(x,\hat x)\leq \rho_X(x,\tilde x)+\rho_X(\tilde x,\hat x)\leq \rho(x,\tilde x)+\kappa\rho_Y(\yb,\tilde y)+\gamma\\
&\leq&\sqrt{1+\kappa^2}\sqrt{(\rho_X(x,\tilde x))^2+(\rho_Y(\yb,\tilde y))^2}+\gamma=\sqrt{1+\kappa^2}\rho_{X\times Y}((x,\yb),(\tilde x,\tilde y))+\gamma\\
&\leq&\sqrt{1+\kappa^2}(L\rho_P(p,\pb)+\beta)+\gamma.
\end{eqnarray*}
It remains to notice again that suitable points $(\tilde x,\tilde y)$ and $\hat x$ can be found for arbitrarily small values of $\beta$ and $\gamma$ whence
\[\distSp{X}{x,S(\pb)}\leq \sqrt{1+\kappa^2}L\rho_P(p,\pb).\]
Since $p\in\B_P(\pb,r_p)$ was arbitrarily fixed, the claimed calmness of $S$ at $(\pb,\xb)$ follows.
\end{proof}

\begin{remark}
  Note that the restricted calmness property with respect to $p$  of $M$  is strictly less stringent than condition (3.3) in \cite[Theorem 3.1]{K3}. This condition was used there together with a sufficient condition for metric subregularity of $\M_{\pb}$ to show the calmness of $S$.
\end{remark}


\begin{corollary}
Consider the implicit multifunction $S$ given by (\ref{eq-112}) with $\yb=0$ and
\[
M(p,x)= A(p)x+b(p)+Q(x)
\]
around the reference point $(\pb ,\xb )\in \Gr S$. Assume that the mappings $A: \mathbb{R}^{l} \to \mathbb{R}^{m \times n}$ and $b: \mathbb{R}^{l}\rightarrow \mathbb{R}^{m}$ are Lipschitz near $\pb $ and the graph of $Q: \mathbb{R}^{n} \rightrightarrows \mathbb{R}^{m}$ is a union of finitely many convex polyhedra. Then $S$ is calm at  $(\pb ,\xb ) $.
\end{corollary}
\begin{proof}
As mentioned in Remark 1, such mapping $M$ has the restricted calmness property with respect to $p$  at  $(\pb ,\xb ,0)$. Furthermore, $M_{\pb }$ is polyhedral and hence metrically subregular at $(\xb ,0)$ by virtue of \cite[Proposition 1]{Ro81}. The statement thus directly follows from Theorem \ref{ThGenClm}.
\end{proof}

In the above way one can model solution maps to parameterized quadratic programs with the parameter arising in the objective.

Next section is devoted to workable conditions ensuring that the assumptions of Theorem \ref{ThGenClm} are fulfilled.

\subsection{\label{SubSec3.2}Calmness criteria}

The next theorem states a sufficient condition for  the restricted calmness property with respect to $p$  of $M$ based on generalized differentiation. From now on  $P=\mathbb{R}^{l}$, $X = \mathbb{R}^{n}$, $Y = \mathbb{R}^{m}$ and $\yb=0$.

\begin{theorem}\label{ThRPCP}
Let $0\in M(\pb,\xb)$ and assume that there do not exist elements $u\in \Sp_{\R^n}$, $q^\ast\in\Sp_{\R^l}$ and sequences $(q_k,u_k,v_k)\to(0,u,0)$, $(q_k^\ast,u_k^\ast, v_k^\ast)\to (q^\ast,0,0)$, $t_k\searrow 0$ such that
\begin{eqnarray}
\label{EqRPCPCoDeriv}  &&(q_k^\ast,u_k^\ast)\in \hat D^\ast M((\pb,\xb,0)+t_k(q_k,u_k,v_k))(v_k^\ast)\ \forall k,\\
\label{EqRPCPAlign}  &&q_k\not=0\ \forall k,\quad \lim_{k\to\infty}\skalp{q_k^\ast,\frac{q_k}{\norm{q_k}}}=-1.
\end{eqnarray}
Then $M$ has the restricted calmness property with respect to $p$  at $(\pb,\xb,0)$.
\end{theorem}
\begin{proof}
By contraposition. Assume on the contrary that $M$ does not have the restricted calmness property  with respect to $p$  at $(\pb,\xb,0)$.
Then there are sequences $(p_k,x_k)\to(\pb,\xb)$ such that $(x_k,0)\in H_M(p_k)$ and $\norm{x_k-\xb}\geq \dist{(x_k,0),H_M(\pb)}>k\norm{p_k-\pb}$. Next we denote by  $(\pb_k,\xb_k,\yb_k)$ for each $k$ a global solution of the program
\[\min_{p,x,y}\phi_k(p,x,y):=\norm{p-\pb}+\frac 2k\norm{x-x_k}+\frac1 {\sqrt{k}}
\norm{y}\quad\mbox{subject to}\quad (p,x,y)\in\Gr M.\]
Then we must have $\pb_k\not=\pb$ since otherwise we would obtain
\begin{eqnarray*}\frac 1k \dist{(x_k,0),H_M(\pb)}&\leq& \frac 1k(\norm{(x_k-\xb_k}+\norm{\yb_k})\leq \frac 2k\norm{\xb_k-x_k}+\frac1{\sqrt{k}}\norm{\yb_k}\\
&=&\phi_k(\pb,\xb_k,\yb_k)=\phi_k(\pb_k,\xb_k,\yb_k)\leq \phi_k(p_k,x_k,0)=\norm{p_k-\pb}
\end{eqnarray*}
contradicting $\dist{(x_k,0),H_M(\pb)}>k\norm{p_k-\pb}$. Hence $t_k:=\norm{(\pb_k,\xb_k,\yb_k)-(\pb,\xb,0)}>0$ and by passing to a subsequence if necessary we can assume that $((\pb_k,\xb_k,\yb_k)-(\pb,\xb,0))/t_k$ converges to some element $(q,u,v)\in\Sp_{\R^l\times\R^n\times\R^m}$. Since $\phi_k(\pb_k,\xb_k,\yb_k)\leq \phi_k(p_k,x_k,0)=\norm{p_k-\pb}$, we can conclude $\norm{\pb_k-\pb}\leq \norm{p_k-\pb}$, $\frac 2k\norm{\xb_k-x_k}\leq \norm{p_k-\pb}$ and $\frac1{\sqrt k}\norm{\yb_k}\leq \norm{p_k-\pb}$, yielding, together with $\norm{p_k-\pb}<\frac 1k\norm{x_k-\xb}$, the relations
 \[\norm{\xb_k-\xb}\geq \norm{x_k-\xb}-\norm{\xb_k-\xb}\geq \norm{x_k-\xb}-\frac k2\norm{p_k-\pb}> \frac 12\norm{x_k-\xb}> \frac k2\norm{p_k-\pb}\geq \frac k2\norm{\pb_k-\pb}\]
and $\norm{\yb_k}\leq \sqrt{k}\norm{p_k-\pb}<\frac 1{\sqrt{k}}\norm{x_k-\xb}<\frac 2{\sqrt{k}}\norm{\xb_k-\xb}$. Hence we can conclude $\norm{\pb_k-\pb}/\norm{\xb_k-\xb}\to 0$, $\norm{\yb_k}/\norm{\xb_k-\xb}\to 0$ and $q=0$, $v=0$ follows. Since
\[\norm{\xb_k-\xb}\leq \norm{x_k-\xb}+\norm{\xb_k-\xb}\leq \norm{x_k-\xb}+\frac k2\norm{p_k-\pb}< \frac 32\norm{x_k-\xb}\to 0,\]
it also follows that $t_k\searrow 0$.\\
Next we utilize the optimality condition $0\in\partial \phi_k(\pb_k,\xb_k,\yb_k)+N_{\Gr M}(\pb_k,\xb_k,\yb_k)$, see \cite[Theorem 8.15]{RoWe98}, where $\partial \phi_k$ can be taken as the subdifferential of convex analysis since $\phi_k$ is convex. Let $(\alpha_k^\ast,\beta_k^\ast,\gamma_k^\ast)\in(-\partial \phi_k(\pb_k,\xb_k,\yb_k))\cap N_{\Gr M}(\pb_k,\xb_k,\yb_k)$. Then, standard arguments from convex analysis yield $\alpha_k^\ast=-(\pb_k-\pb)/\norm{\pb_k-\pb}$ 
 and
$\beta_k^\ast\in\frac 2k \B_{\R^n}$, $\gamma_k^\ast\in\frac 1{\sqrt{k}}\B_{\R^m}$
and we deduce $(\beta_k^\ast,\gamma_k^\ast)\to (0,0)$ as $k\to\infty$. By the definition of the limiting normal cone we can find for each $k$ some elements $(\pb_k',\xb_k',\yb_k')$ and $(q_k^\ast,u_k^\ast,-v_k^\ast)\in \hat N_{\Gr M}(\pb_k',\xb_k',\yb_k')$ verifying
\[\norm{(\pb_k',\xb_k',\yb_k')-(\pb_k,\xb_k,\yb_k)}\leq \norm{\pb_k-\pb}/k,\quad \norm{(q_k^\ast,u_k^\ast,-v_k^\ast)-(\alpha_k^\ast,\beta_k^\ast,\gamma_k^\ast)}\leq \frac 1k.\]
Then we obviously have $(u_k^\ast,v_k^\ast)\to (0,0)$.
Since $\alpha_k^\ast\in\Sp_{\R^l}$ $\forall k$ and $q_k^\ast-\alpha_k^\ast\to0$, by  passing to a subsequence if necessary we can assume that $q_k^\ast$ converges to some $q^\ast\in\Sp_{\R^l}$. Now set $(q_k,u_k,v_k):=((\pb_k',\xb_k',\yb_k')-(\pb,\xb,\yb))/t_k$. Since $\norm{\pb_k-\pb}\leq t_k$, the choice of $(\pb_k',\xb_k',\yb_k')$ guarantees that $(q_k,u_k,v_k)$ converges to $(0,u,0)$.
Further we have
\begin{eqnarray*}\left\Vert \frac{\pb_k-\pb}{\norm{\pb_k-\pb}}-\frac{\pb_k'-\pb}{\norm{\pb_k'-\pb}}\right\Vert&=&
\left\Vert\frac{(\pb_k'-\pb)(\norm{\pb_k'-\pb}-\norm{\pb_k-\pb})-(\pb_k'-\pb_k)\norm{\pb_k'-\pb}}{\norm{\pb_k'-\pb}\cdot\norm{\pb_k-\pb}}\right\Vert\\
&\leq& \frac{2\norm{\pb_k'-\pb_k}\cdot\norm{\pb_k'-\pb}}{\norm{\pb_k'-\pb}\cdot\norm{\pb_k-\pb}}\leq \frac 2k
\end{eqnarray*}
implying
\[\lim_{k\to\infty}\skalp{q_k^\ast,\frac{q_k}{\norm{q_k}}}=\lim_{k\to\infty}\skalp{q_k^\ast,\frac{\pb_k'-\pb}{\norm{\pb_k'-\pb}}}
=\lim_{k\to\infty}\skalp{q_k^\ast,\frac{\pb_k-\pb}{\norm{\pb_k-\pb}}}=\lim_{k\to\infty}\skalp{\alpha_k^\ast,\frac{\pb_k-\pb}{\norm{\pb_k-\pb}}}=-1.\]
Finally note that we have $(q_k^\ast,u_k^\ast)\in\hat D^\ast M((\pb,\xb,0)+t_k(q_k,u_k,v_k))(v_k^\ast)$ by our construction. Hence we see that $q^\ast,u$ together with the sequences $t_k$, $(q_k,u_k,v_k)$ and $(q_k^\ast,u_k^\ast,v_k^\ast)$ violate the assumptions of the theorem yielding the desired contradiction. This finishes the proof.
\end{proof}

Condition \eqref{EqRPCPCoDeriv} suggests the following definition.
\begin{definition}The  {\em outer coderivative of $M$ with respect to $p$ in direction $u$ at $(\pb,\xb,0)$}  is
the multifunction $D^\ast_{>_p}M((\pb,\xb,0);u):\R^m\rightrightarrows\R^l\times\R^n$, where  $D^\ast_{>_p}M((\pb,\xb,0);u)(v^\ast)$ consists of all pairs $(q^\ast,u^\ast)$ such that there are sequences $t_k\searrow 0$, $(q_k,u_k,v_k)\to(0,u,0)$, $(q_k^\ast,u_k^\ast,v_k^\ast)\to (q^\ast,u^\ast,v^\ast)$ verifying
\[ q_k\not=0\quad\mbox{and}\quad(q_k^\ast,u_k^\ast,-v_k^\ast)\in\widehat N_{\Gr M}(\pb+t_kq_k,\xb+t_ku_k,t_kv_k)\ \forall k.\]
\end{definition}

By the definition of the directional limiting coderivative we see that
\begin{equation}\label{EqInclOuterCoderiv}D^\ast_{>_p}M((\pb,\xb,0);u)(v^\ast)\subset D^\ast M((\pb,\xb,0);(0,u,0))(v^\ast)\ \forall v^\ast\in\R^m.\end{equation}
Further we have $D^\ast_{>_p}M((\pb,\xb,0);u)\equiv\emptyset$ whenever $(0,u,0)\not\in T_{\Gr M}(\pb,\xb,0)$, i.e., $0\not\in DM(\pb,\xb,0)(0,u)$.
These observations yield the following point based sufficient condition for the restricted calmness property with respect to $p$  to hold.
\begin{corollary}\label{CorPropC}
Let $0\in M(\pb,\xb)$ and assume that there do not exist elements $u\in \Sp_{\R^n}$ and $q\in \Sp_{\R^l}\cap T_{\dom H_M}(\pb)$ satisfying
\begin{eqnarray}
\label{EqRPCPTan}&&0\in DM(\pb,\xb,0)(0,u),\\
\label{EqRPCPCoderiv2}  &&(-q,0)\in D^\ast_{>_p} M((\pb,\xb,0);u)(0).
\end{eqnarray}
Then $M$ has the restricted calmness property with respect to $p$  at $(\pb,\xb,0)$.
\end{corollary}
\begin{proof}Consider the sequences specified in Theorem \ref{ThRPCP} which satisfy, in particular, the relations (\ref{EqRPCPCoDeriv}), (\ref{EqRPCPAlign}).   By passing to a subsequence if necessary we can assume that $q_k/\norm{q_k}$ converges to some $q\in\Sp_{\R^l}$. Since we also have $\pb+(t_k\norm{q_k})(q_k/\norm{q_k})\in\dom H_M$ and $t_k\norm{q_k}\to 0$, the inclusion $q\in T_{\dom H_M}(\pb)$ follows. From the second condition in (\ref{EqRPCPAlign}) it follows that
\[ \langle q^{\ast}, q\rangle = -1 \mbox{ with } q^{\ast}, q \in S_{\mathbb{R}^{l}}.\]
This is, however, possible only when $q^{\ast}= -q$  and we are done.
\end{proof}

In the following example we demonstrate the application of Theorem \ref{ThRPCP} in a situation when $M$ is of the form $M(x,y)=G(x,y)+Q(y)$ with $G$ being non-Lipschitzian.
\begin{example}
Consider the multifunction $M:\R\times\R\to\R$ given by $M(p,x)=\sqrt{\vert px\vert}+x+\R_+$ at $(\pb,\xb)=(0,0)$. Straightforward calculations yield
\[T_{\Gr M}(0,0,0)=\{(q,u,v)\mv v\geq \sqrt{\vert qu\vert}+u\}\]
and in case when $px\not=0$,
\[\hat D^\ast M(p,x,y)(v^\ast)=\begin{cases}\{(0,0)\}&\mbox{if $y>\sqrt{\vert px\vert}+x$, $v^\ast=0$},\\
\{ (\sqrt{\vert x/p\vert} v^\ast{\rm sign\,} p,\sqrt{\vert p/x\vert}v^\ast{\rm sign\,} x+v^\ast)\}&\mbox{if $y=\sqrt{\vert px\vert}+x$, $v^\ast\geq0$},\\
\emptyset&\mbox{else.}\end{cases}\]
Now assume that there are sequences $t_k\searrow 0$, $(q_k,u_k,v_k)\to (0,u,0)\in T_{\Gr M}(0,0,0)$ and $(q_k^\ast,u_k^\ast,v_k^\ast)\to (q^\ast,0,0)$ such that $u\in\Sp_{\R}$, $q^\ast\in\Sp_{\R}$ and \eqref{EqRPCPCoDeriv}, \eqref{EqRPCPAlign} hold. The condition $(0,u,0)\in T_{\Gr M}(0,0,0)$ amounts to $u\leq 0$, together with $u\in\Sp_{\R}$ we obtain $u=-1$ and therefore $u_k\not=0$ holds for all $k$ sufficiently large. Since $q_k\not=0$, we obtain $v_k^\ast\geq 0$,
\[q_k^\ast=\sqrt{\left\vert \frac{t_ku_k}{t_kq_k}\right\vert} v_k^\ast{\rm  sign\,}{t_kq_k}=\sqrt{\left\vert \frac{u_k}{q_k}\right\vert} v_k^\ast {\rm sign\,}{q_k}\]
and consequently
\[\skalp{q_k^\ast,\frac{q_k}{\vert q_k\vert}}=\sqrt{\left\vert \frac{u_k}{q_k}\right\vert} v_k^\ast\geq 0\]
contradicting \eqref{EqRPCPAlign}. Hence, it follows from Theorem \ref{ThRPCP} that $M$ has the restricted calmness property with respect to $p$  at $(\pb,\xb,0)$.\hfill$\triangle$
\end{example}

Corollary \ref{CorPropC} is illustrated by the following example.
\begin{example}
  Let $M(p,x)=\hat N_{\Gamma(p)}(x)$, where  $\Gamma(p)=\{x\in \R\mv px\leq 0\}$ for $p\in\R$, and let $(\pb,\xb)=(0,0)$. Then
  \[\Gamma(p)=\begin{cases}
    \R&\mbox{if $p=0$,}\\
    \R_-&\mbox{if $p>0$,}\\
        \R_+&\mbox{if $p<0$}
  \end{cases}\]
  and therefore
  \[M(p,x)=\begin{cases}
    \{0\}&\mbox{if $p=0$ or $px<0$,}\\
    \R_+&\mbox{if $p>0$, $x=0$,}\\
    \R_-&\mbox{if $p<0$, $x=0$,}\\
    \emptyset&\mbox{else,}
  \end{cases}\qquad H_M(p)=
  \begin{cases}
    \R\times\{0\}&\mbox{if $p=0$,}\\
    (\{0\}\times\R_+)\cup(\R_-\times\{0\})&\mbox{if $p>0$,}\\
    (\{0\}\times\R_-)\cup(\R_+\times\{0\})&\mbox{if $p<0$.}
  \end{cases}\]
  Hence $M$ has the restricted calmness property with respect to $p$  at $(\pb,\xb,0)$, but $H_M$ is not calm at $(\pb,(\xb,0))$. Now let us consider the conditions of Corollary \ref{CorPropC}. Since \eqref{EqRPCPCoderiv2} involves the outer directional coderivative of $M$ with respect to $p$ only in directions $u\not=0$, we have to compute the regular normal cone to $\Gr M$ at points $(p,x,y,z)$ with $px\not=0$.
  Straightforward calculations yield
  \[\widehat N_{\Gr M}(p,x,y)=\begin{cases}\{0\}\times\{0\}\times\R&\mbox{if $y=0$ and $px<0$,}\\
  \emptyset&\mbox{else if $px\not=0$,}\end{cases}\]
  showing $D^\ast_{>_p} M((\pb,\xb,0);u)(v^\ast)=\{(0,0)\}$ $\forall v^\ast\in\R$ provided $u\not=0$. Hence we can also conclude from Corollary \ref{CorPropC} that $M$ has the restricted calmness property with respect to $p$. Note that in this example the inclusion \eqref{EqInclOuterCoderiv} is proper and one could not detect the restricted calmness property with respect to $p$  of $M$ by using the standard directional limiting coderivative.\hfill$\triangle$
\end{example}

To verify the metric subregularity of the mapping $M_{\pb}$ at $(\xb,0)$ one can use the criteria presented in Section 2. In the following theorem we apply these conditions  to the frequently arising case of the  so-called parameterized constraint systems. We prove that we  obtain not only calmness of the solution mapping $S$ but also, under some  suitable assumptions, the non-emptiness of $S(p)$ near $\pb $.
\begin{theorem}\label{ThClmConstrSystFO}
  Let
  \begin{equation}\label{EqParamConstrSyst}M(p,x)=G(p,x)-D,\end{equation}
  where $D\subset \R^m$ is closed, $G$ maps $\R^l\times\R^n$ into $\R^m$ and consider the reference point $(\pb,\xb)\in G^{-1}(D)$. Assume that $G(\pb,\cdot)$ is  strictly differentiable at $\xb$ and there are  neighborhoods $W$ of $\pb$, $U$ of $\xb$ and a real $L'$ such that
  \begin{equation}\label{EqPropC}\norm{G(p,x)-G(\pb,x)}\leq L'\norm{p-\pb}, \forall (p,x)\in W\times U.\end{equation}
  If there do not exist vectors $0\not=u\in\R^n$, $0\not=v^\ast\in\R^m$ such that
  \begin{eqnarray}\label{EqSubreg1a}\nabla_x G(\pb,\xb)u&\in& T_D(G(\pb,\xb))\\
  \label{EqSubreg2a}0&=&\nabla_x G(\pb,\xb)^Tv^\ast,\\
  \label{EqSubreg3a}v^\ast&\in& N_D(G(\pb,\xb);\nabla_x G(\pb,\xb)u),
  \end{eqnarray}
  then $S$ is calm at $(\pb,\xb)$.

    If in addition $G$ is  partially differentiable with respect to $x$ on $W\times U$, if the partial derivative $\nabla_x G$ is continuous at $(\pb,\xb)$, if  for every $p\in W$ the mapping $\nabla_x G(p,\cdot)$ is continuous on $U$
   and if there exists some  $\tilde u\in\Sp_{\R^n}$ such that $\nabla_x G(\pb,\xb)\tilde u\in T_D(G(\pb,\xb))$ and $\nabla_x G(\pb,\xb)\tilde u$ is derivable, then there exists a neighborhood $\tilde W$ of $\pb$ and a real $\tilde L$ such that $S(p)\not=\emptyset$ $\forall p\in\tilde W$ and
  \begin{equation}\label{EqLip1}
  \dist{\xb,S(p)}\leq \tilde L\norm{p-\pb},\ p\in\tilde W.
  \end{equation}
\end{theorem}
\begin{proof}
  We first show that $M$ has the restricted calmness property with respect to $p$  at $(\pb,\xb,0)$. Indeed, if $p\in W$ and $(x,0)\in H_M(p)\cap U\times\{0\}$, then $0\in M(p,x)=G(p,x)-D$
  and consequently $G(\pb,x)-G(p,x)\in G(\pb,x)-D$, i.e., $(x,G(\pb,x)-G(p,x))\in H_M(\pb)$. Hence
  \[\dist{(x,0),H_M(\pb)}\leq \norm{(x,0)-(x,G(\pb,x)-G(p,x))}\leq\norm{G(\pb,x)-G(p,x)}\leq L'\norm{p-\pb}\]
  and the restricted calmness property with respect to $p$  for $M$ at $(\pb,\xb,\yb)$ follows.
  By using FOSCMS of Corollary \ref{CorMS_ConstrSyst} we see that the imposed assumptions guarantee metric subregularity of $M_{\pb}$ at $(\xb,0)$. Thus  calmness of $S$ follows from Theorem \ref{ThGenClm}.  It remains to show  the non-emptiness of $S$ near $\pb$ and the bound \eqref{EqLip1}. This is done by contraposition. Assume on the contrary that there is some sequence $p_k$ converging to $\pb$ such that  $p_k\not=\pb$ and
  \begin{equation}\label{eq-2000}
\dist{\xb, S(p_k)}>k\norm{p_k-\pb}.
  \end{equation}
   Since the tangent vector $\nabla_x G(\pb,\xb)\tilde u$ is assumed to be derivable, there exists a mapping $\xi:\R_+\to D$ such that $\xi(0)=G(\pb,\xb)$ and $\xi(t)-(G(\pb,\xb)+t\nabla_x G(\pb,\xb)\tilde u)=\oo(t)$ as $t\searrow 0$. Since $G(\pb,\cdot)$ is assumed to be continuously differentiable, $\eta(t):= \norm{G(\pb,\xb+t\tilde u)-\xi(t)}=\oo(t)$ follows. By passing to a subsequence if necessary we can assume that $\norm{p_k-\pb}\leq k^{-2}$ and $\eta(t_k)\leq \norm{p_k-\pb}$ holds for all $k$, where $t_k:=\frac{k\norm{p_k-\pb}}{2}$. Then $t_k\searrow 0$ and we can find some $L>0$ such that $\norm{G(p_k,\xb+t_k\tilde u)-G(\pb,\xb+t_k\tilde u)}\leq L\norm{p_k-\pb}$ holds for all $k$ sufficiently large, without loss of generality for all $k$.

  Next consider for each $k$ a solution $(\xb_k,\yb_k)$ of the program
  \[\min_{x,y}\phi_k(x,y):=\frac {4(L+1)}{k}\norm{x-(\xb+t_k\tilde u)}+\norm{y}\quad\mbox{subject to}\quad (x,y)\in\Gr M_{p_k}.\]
  Because of $(\xb+t_k\tilde u, G(p_k,\xb+t_k\tilde u)-\xi(t_k))\in \Gr M_{p_k}$ we obtain
  \begin{eqnarray*}\phi_k(\xb_k,\yb_k)&\leq& \norm{G(p_k,\xb+t_k\tilde u)-\xi(t_k)}\\
  &\leq&   \norm{G(p_k,\xb+t_k\tilde u)-G(\pb,\xb+t_k\tilde u)}+\norm{G(\pb,\xb+t_k\tilde u)-\xi(t_k)}\\
  &\leq& L\norm{p_k-\pb}+\eta(t_k)\leq (L+1)\norm{p_k-\pb}.
  \end{eqnarray*}
  Further we must have $\yb_k\not=0$ since otherwise we would have $\frac {4(L+1)}{k}\norm{\xb_k-(\xb+t_k\tilde u)}=\phi_k(\xb_k,\yb_k)\leq (L+1)\norm{p_k-\pb}$ and $0=\yb_k\in M(p_k,\xb_k)$ implying $\xb_k\in S(p_k)$ and
  \[\dist{\xb, S(p_k)}\leq \norm{\xb_k-\xb}\leq \norm{\xb_k-(\xb+t_k\tilde u)}+t_k\norm{\tilde u}\leq \frac k4\norm{p_k-\pb}+\frac k2\norm{p_k-\pb},\]
  which clearly contradicts the  inequality (\ref{eq-2000}).

  By the first order optimality conditions \cite[Theorem 8.15]{RoWe98} we have $0\in \partial \phi_k(\xb_k,\yb_k) +N_{\Gr M_{p_k}}(\xb_k,\yb_k)$, where $\partial \phi_k$ stands for the subdifferential in the sense of convex analysis, cf. \cite[Proposition 8.12]{RoWe98}. Hence there are elements $(x_k^\ast,y_k^\ast)\in-\partial \phi_k(\xb_k,\yb_k)\cap N_{\Gr M_{p_k}}(\xb_k,\yb_k)$. Then $x_k^\ast\in \frac {4(L+1)}k \B_{\R^n}$ and $y_k^\ast\in\Sp_{\R^m}$ because of $\yb_k\not=0$. By the definition of the limiting normal cone we can find elements $(\xb_k',\yb_k')\in \Gr M_{p_k}$ and $(u_k^\ast,v_k^\ast)\in \widehat N_{\Gr M_{p_k}}(\xb_k',\yb_k')$ such that $\norm{(\xb_k,\yb_k)-(\xb_k',\yb_k')}\leq t_k/k$ and $\norm{(x_k^\ast,y_k^\ast)-(u_k^\ast,v_k^\ast)}\leq 1/k$. From $M_{p_k}(\cdot)=G(p_k,\cdot)-D$ we conclude that $-v_k^\ast\in\widehat N_D(G(p_k,\xb_k')-\yb_k')$ and $u_k^\ast=-\nabla_x G(p_k,\xb_k')^Tv_k^\ast\to 0$. Since
  \[\norm{\xb_k'-\xb}\geq \norm{\xb_k-\xb}-\frac{t_k}k\geq t_k\norm{\tilde u}-\norm{\xb_k-(\xb+t_k\tilde u)}-\frac{t_k}k\geq t_k-\frac k4\norm{p_k-\pb}-\frac {t_k}k=(\frac k4-\frac 12)\norm{p_k-\pb}\]
  and
  \[\norm{\xb_k'-\xb}\leq \norm{\xb_k-\xb}+\frac{t_k}k\leq t_k\norm{\tilde u}+\norm{\xb_k-(\xb+t_k\tilde u)}+\frac{t_k}k\leq t_k+\frac k4\norm{p_k-\pb}+\frac {t_k}k=(\frac {3k}4+\frac 12)\norm{p_k-\pb},\]
  it follows that $\tau_k\to 0$ and $(p_k-\pb)/\tau_k\to 0$, where $\tau_k:=\norm{\xb_k'-\xb}$.  By passing to a subsequence we may assume that the sequence $-v_k^\ast$ converges to $v^\ast$ and the sequence $(\xb_k'-\xb)/\tau_k$ converges to some $u\in\Sp_{\R^n}$. Together with
  \[\norm{\yb_k'}\leq\norm{\yb_k}+\frac{t_k}k\leq \phi_k(\xb_k,\yb_k)+\frac{t_k}k\leq (L+\frac 32)\norm{p_k-\pb}=\oo(\tau_k),\]
  we obtain
  \begin{eqnarray*}\lim_{k\to\infty}\frac{G(p_k,\xb_k')-\yb_k'-G(\pb,\xb)}{\tau_k}&=&
  \lim_{k\to\infty}\left(\frac{G(p_k,\xb_k')-G(\pb,\xb_k')}{\tau_k}+\frac{G(\pb,\xb_k')-G(\pb,\xb)}{\tau_k}\right)\\
  &=&\nabla_x G(\pb,\xb)u.\end{eqnarray*}
  This shows $v^\ast\in N_D(G(\pb,\xb),\nabla_x G(\pb,\xb)u)$ and $\nabla_x G(\pb,\xb)u\in T_D(G(\pb,\xb))$. Since we also have
  \[0=\lim_{k\to\infty}\nabla_xG(p_k,\xb_k')^T (-v_k^\ast)=\nabla_xG(\pb,\xb)^T v^\ast\]
  we obtain a contradiction to \eqref{EqSubreg1a}-\eqref{EqSubreg3a} and this completes the proof.
\end{proof}

The next statement concerns the frequently arising case when $D$ is a union of convex polyhedrons which occurs, e.g., in case of parameterized complementarity problems.

\begin{theorem}\label{ThClmConstrSystSO}
  In the setting of Theorem \ref{ThClmConstrSystFO} consider the situation that  $D\subset \R^m$ is  the union of finitely many convex polyhedra, $G(\pb,\cdot)$ is  twice Fr\'echet differentiable at $\xb$ and there are  neighborhoods $W$ of $\pb$, $U$ of $\xb$ and a real $L'$ such that \eqref{EqPropC} holds.

  If there do not exist vectors $0\not=u\in\R^n$, $0\not=v^\ast\in\R^m$ verifying \eqref{EqSubreg1a}-\eqref{EqSubreg3a} and
  \begin{equation}\label{EqSubreg4a}u^T\nabla_{xx}^2 ({v^\ast}^TG)(\pb,\xb)u\geq 0,
  \end{equation}
  then $S$ is calm at $(\pb,\xb)$.

  Moreover, if in addition $G$ is twice partially differentiable with respect to $x$ on $W\times U$, if $G$, $\nabla G$ and $\nabla_{xx}^2 G$ are continuous on $W\times U$, if
  \[\norm{\nabla_x G(p,\xb)-\nabla_x G(\pb,\xb)}\leq L'\norm{p-\pb}\]
  holds for all $p$ near $\pb$
   and if  there exists some nonzero $\tilde u$ with $\nabla_x G(\pb,\xb)u\in T_D(G(\pb,\xb))$, then there exist a neighborhood $\tilde W$ of $\pb$ and a real $\tilde L$ such that $S(p)\not=\emptyset$ $\forall p\in\tilde W$ and
  \begin{equation}\label{EqHoeld1}
  \dist{\xb,S(p)}\leq \tilde L\norm{p-\pb}^{1/2},\ p\in\tilde W.
  \end{equation}
\end{theorem}
\begin{proof}The same arguments as in the proof of Theorem \ref{ThClmConstrSystFO} show that $M$ has the restricted calmness property with respect to $p$  at $(\pb,\xb,0)$. By our assumptions, SOSCMS of Corollary \ref{CorMS_ConstrSyst} is fulfilled for $M_{\pb}$ at $(\xb,0)$ and therefore the calmness of $S$ at $(\pb,\xb)$ follows from Theorem \ref{ThGenClm}.
  Further,  the non-emptiness   of $S(p)$ and the bound \eqref{EqHoeld1} follow from  \cite[Proposition 2 (2.)]{GfrKl15}.
\end{proof}

The situation of Theorem \ref{ThClmConstrSystFO} is illustrated in the following example.
 \begin{example}\label{ExMetrSubreg}
 Let $p \in \mathbb{R}^{2}, x \in \mathbb{R}^{2} $  and $S$ be implicitly given by the complementarity problem
 \[
 0 \leq x_{1}-p_{1}\perp x_{2}-p_{2}\geq 0
 \]
 combined with the (nonlinear) inequality constraints $-x_{1}-x_{1}^{2} \leq x_{2}\leq x_{1}+x_{1}^{2}$. Let $(\pb ,\xb )=(0_{\mathbb{R}^{2}}, 0_{\mathbb{R}^{2}})$. This problem attains the form \eqref{EqParamConstrSyst} with
\begin{equation}\label{eq-1}
G(p,x)=
\left[\begin{array}{l}
-p_{1} +x_{1}\\
-p_{2}+x_{2}\\
-x_{1}-x^{2}_{1}-x_{2}\\
-x_{1}-x^{2}_{1}+ x_{2}
\end{array}\right]
\end{equation}
and $D = \mathcal{K}\times \mathbb{R}^{2}_{-}$, where $\mathcal{K}$
denotes the ``complementarity angle'', i.e.,
\[
\mathcal{K}:= \{z \in \mathbb{R}^{2}_{+}| z_{1}z_{2}= 0\}.
\]
$G$ clearly fulfills all the assumptions of Theorem \ref{ThClmConstrSystFO}.
It  follows from \eqref{EqSubreg1a} that  we have to analyze the directions $u \in \mathbb{R}^{2}$ satisfying the conditions
\[
(u_{1}, u_{2})\in \mathcal{K}, -u_{1}-u_{2}\leq 0, -u_{1}+u_{2}\leq 0,
\]
which amount to $u_{1}\geq 0, u_{2}=0$. As to condition \eqref{EqSubreg2a}, we obtain the equalities
\begin{equation}\label{eq-22}
\begin{array}{l}
 v^{*}_{1}- v^{*}_{3}-v^{*}_{4}=0\\
 v^{*}_{2}- v^{*}_{3}+v^{*}_{4}=0.
\end{array}
\end{equation}
Finally we observe that for any sequence of vectors $u^{(k)}\rightarrow (\bar u_1, 0)$ with $\bar u_1 > 0$ and for any sequence of reals $t^{(k)}\searrow 0$ such that
\[
G(\pb ,\xb )+t^{(k)}\nabla_{x}G(\pb ,\xb )u^{(k)} = t^{(k)}
\left[ \begin{array}{l}
u_{1}^{(k)}\\
u_{2}^{(k)}\\
-u_{1}^{(k)}-u_{2}^{(k)}\\
-u_{1}^{(k)}+u_{2}^{(k)}
\end{array}\right] \in \mathcal{K} \times \mathbb{R}^{2}_{-},
\]
one has $u_2^{(k)}=0$ and therefore
\begin{equation}\label{eq-23}
N_{\mathcal{K} \times \mathbb{R}^{2}_{-}}(G(\pb ,\xb )+t^{k}\nabla_{x} G(\pb ,\xb )u^{(k)})= \{v^{*}\in \mathbb{R}^{4}| v^{*}_{1}=v^{*}_{3}=v^{*}_{4}=0\}.
\end{equation}
Thus, by combining \eqref{eq-22} and \eqref{eq-23} we conclude that $v^{*}_{2}=0$ as well and, due to \eqref{EqSubreg3a}, the corresponding implicit multifunction $S$ is calm at $(\pb ,\xb )$. Since the direction $\tilde u=(1,0)$ fulfills $\nabla_x G(\pb,\xb)\tilde u=(1,0,-1-1)\in D =T_D(G(\pb,\xb))$ and $\nabla_x G(\pb,\xb)\tilde u$ is derivable, we also have $S(p)\not=\emptyset$ and $\dist{\xb,S(p)}\leq \tilde L\norm{p-\pb}$ for some real $\tilde L$ and all $p$ near $\pb$.

Note that in the above example the implicit multifunction $S$ does not possess the Aubin property around $(\pb ,\xb )$, because  we have
\[(0,0)\in S(0,p_2)\ \forall p_2<0,\qquad (0,0)\not\in S(p_1,p_2)\ \forall p_1>0,\ p_2<0.\]
Further, $S$ does not have  the isolated calmness property at $(\pb ,\xb )$ as well because $S(\pb )=\mathbb{R}_{+}\times \{0\}$. Moreover, $M_{\pb}$ fulfills FOSCMS but is neither strongly metrically subregular nor metrically regular.
\hfill$\triangle$
\end{example}

The next example illustrates the situation of Theorem \ref{ThClmConstrSystSO}.
\begin{example}For $p\in\R$ let $S(p)$ be given by the solutions $x\in\R^2$ of the nonlinear inequalities
\[ p-\frac 12 x_1^2+x_2\leq 0,\qquad p-\frac 12 x_1^2-x_2\leq 0.\]
Again this system can be written in the form \eqref{EqParamConstrSyst} with
\[G(p,x)=\left(\begin{array}
  {c}p-\frac 12 x_1^2+x_2 \\p-\frac 12 x_1^2-x_2
\end{array}\right),\qquad D=\R^2_-.\]
Let $\pb=0$, $\xb=0_{\R^2}$. $D$ is a convex polyhedron and we will apply Theorem \ref{ThClmConstrSystSO}. The conditions \eqref{EqSubreg1a}-\eqref{EqSubreg4a}
amount to
\begin{eqnarray*}
  u_2\leq 0,\ -u_2\leq 0,\ v_1^\ast-v_2^\ast=0,\ v_1^\ast\geq 0,\ v_2^\ast\geq 0,\ -(v_1^\ast+v_2^\ast)u_1^2\geq 0,
\end{eqnarray*}
which cannot be fulfilled with $u=(u_1,u_2)\not=(0,0)$, $v^\ast=(v_1^\ast,v_2^\ast)\not=(0,0)$. Hence, $S$ is calm at $(\pb,\xb)$ and since the direction $\tilde u=(1,0)$ fulfills $\nabla_xG(\pb,\xb)\tilde u=(0,0)\in T_D(G(\pb))$ we conclude that $S(p)\not=\emptyset$ and an estimate of the form $\dist{\xb,S(p)}=O(\sqrt{\vert p\vert})$ holds for $p$ near $0$. Indeed, we have $\dist{\xb,S(p)}=\norm{(0,\pm\sqrt{-p})}=\sqrt{-p}$ for all $p<0$. This example demonstrates the antisymmetry of the calmness property. Although the points $x\in S(p)$ near $\xb$ are close to $S(\pb)$ up to the order $O(\norm{p-\pb})$, the point $\xb\in S(\pb)$ is not close to $S(p)$ with this order.\hfill$\triangle$
\end{example}

Since every inclusion $0\in M(p,x)$ can be written equivalently in the form \eqref{EqParamConstrSyst} by
\begin{equation}\label{EqGenMConstrSyst}
  0\in \tilde M(p,x):=(p,x,0)-\Gr M,
\end{equation}
one can combine Corollary \ref{CorPropC}, Theorem \ref{ThGenClm} and Theorem \ref{ThClmConstrSystFO} to obtain pointbased conditions for the calmness of solution mappings of general inclusions.
\begin{corollary}
  Let  $S(p):=\{x\mv 0\in M(p,x)\}$, where $M:\R^l\times \R^n\to\R^m$ is a closed multifunction, and let $\xb\in S(\pb)$.

   {\bf(i)} Assume that there do not exist directions $u\not=0$, $\hat u\not=0$, $q\in T_{\dom H_M}$ and elements $q^\ast\in\Sp_{\R^l}$, $v^\ast\in\Sp_{\R^m}$ such that
  \begin{eqnarray}
    \label{EqClmGenM1}&& 0\in DM(\pb,\xb,0)(0,u),\quad (q^\ast,0)\in D^\ast_{>_p} M((\pb,\xb,0);u)(0),\quad \skalp{q^\ast,q}=-1,\\
    \label{EqClmGenM2}&& 0\in DM_{\pb}(\xb,0)(\hat u),\quad 0\in D^\ast M_{\pb}((\xb,0);(\hat u,0))(v^\ast).
  \end{eqnarray}
  Then $S$ is calm at $(\pb,\xb)$.

  {\bf (ii)} If there do not exist a direction $u\not=0$ and elements $(q^\ast,v^\ast)\not=(0,0)$ such that
\begin{equation} \label{EqClmGenM3} 0\in DM(\pb,\xb,0)(0,u),\quad (q^\ast,0)\in D^\ast M((\pb,\xb,0);(0,u,0))(v^\ast)
\end{equation}
and if there exists a direction $\tilde u$ such that $(0,\tilde u,0)\in T_{\Gr M}(\pb,\xb,0)$ and $(0,\tilde u,0)$ is derivable, then $S$ is calm at $(\pb,\xb)$ and there exist a real $\tilde L$ and a neighborhood $\tilde W$ of $\pb $ such that
\[S(p)\not=\emptyset,\ \dist{\xb,S(p)}\leq\tilde L\norm{p-\pb}\ \forall p\in\tilde W.\]
\end{corollary}
\begin{proof}
  By Corollary \ref{CorPropC} condition \eqref{EqClmGenM1} ensures that $M$ has the restricted calmness property with respect to $p$  at $(\pb,\xb,0)$. Further, $M_{\pb}$ is metrically subregular at $(\xb,0)$ due to \eqref{EqClmGenM2} and FOSCMS \eqref{EqEquivCrDirLimCoDer}.  Hence, calmness of $S$ follows from Theorem \ref{ThGenClm}. The second statement follows from Theorem \ref{ThClmConstrSystFO} together with the observation that conditions \eqref{EqSubreg1a}-\eqref{EqSubreg3a} applied to $\tilde M$ given by \eqref{EqGenMConstrSyst} amount to \eqref{EqClmGenM3}.
\end{proof}

\section{Aubin property of implicit multifunctions}
The aim of this section is to investigate the Aubin property of $S$ given by \eqref{eq-111} with a closed-graph mapping $M:\R^l\times\R^n\rightrightarrows \R^m$. We start with the following proposition.
\begin{proposition}\label{PropDirNormalCone}
  Let $\M:\R^s\rightrightarrows\R^d$ be a multifunction with closed graph. Given $(\xb,\yb)\in\Gr \M$ and a direction $u\in\R^s$, assume that $\M$ is metrically subregular in direction $u$ at $(\xb,\yb)$ with modulus $\kappa$. Then
  \begin{equation}\label{EqTangDir}
    u\in T_{\M^{-1}(\yb)}(\xb)\quad\Leftrightarrow\quad 0\in D\M(\xb,\yb)(u)
  \end{equation}
  and
  \begin{eqnarray}
  \label{EqInclDirNormal}N_{\M^{-1}(\yb)}(\xb;u)&\subset& \{x^\ast\mv \exists y^\ast\in
  \kappa\norm{x^\ast}\B_{\R^d}:\ (x^\ast,y^\ast)\in N_{\Gr \M}((\xb,\yb);(u,0))\}\\
  \nonumber&\blue{\subset}& {\rm rge} ~D^\ast \M((\xb,\yb);(u,0)).
  \end{eqnarray}
\end{proposition}
\begin{proof}In order to prove the forward implication in \eqref{EqTangDir}, assume $u\in T_{\M^{-1}(\yb)}(\xb)$. Then there are sequences $t_k\searrow 0$ and $u_k\to u$  such that $\xb+t_ku_k\in \M^{-1}(\yb)$ or, equivalently, $\yb=\yb+t_k0\in \M(\xb+t_ku_k)$, implying $(u,0)\in T_{\Gr \M}(\xb,\yb)$ which is the same as
$0\in D\M(\xb,\yb)(u)$. Conversely, if $0\in D\M(\xb,\yb)(u)$, then there are sequences $t_k\searrow 0$ and $(u_k,v_k)\to(u,0)$ such that $\yb+t_kv_k\in\M(\xb+t_ku_k)$.
By virtue of the assumed metric subregularity of $\M$ in direction $u$ 
and choosing $\kappa'>\kappa$ we have
\[\dist{\xb+t_ku_k,\M^{-1}(\yb)}\leq\kappa'\dist{\yb,\M(\xb+t_ku_k)}\leq \kappa' t_k\norm{v_k}\]
for all $k$ sufficiently large and therefore we can find a sequence $x_k\in \M^{-1}(\yb)$ verifying $\norm{x_k-(\xb+t_ku_k)}\leq \kappa' t_k\norm{v_k}$. Thus
\[\lim_{k\to\infty}\norm{\frac{x_k-\xb}{t_k}-u}\leq \lim_{k\to\infty}\norm{\frac{x_k-\xb}{t_k}-u_k}+\norm{u_k-u}\leq \lim_{k\to\infty}\kappa'\norm{v_k}+\norm{u_k-u}=0\]
and we conclude $u\in T_{\M^{-1}(\yb)}(\xb)$. Hence the relation \eqref{EqTangDir} is shown.

To prove  inclusion \eqref{EqInclDirNormal}, consider $x^\ast\in N_{\M^{-1}(\yb)}(\xb;u)$. Then there are sequences $t_k\searrow 0$, $u_k\to u$ and $x_k^\ast\to x^\ast$ such that $x_k^\ast\in\widehat N_{\M^{-1}(\yb)}(x_k)$, where $x_k:=\xb+t_ku_k$. Hence, for each $k$ we can find some radius $r_k>0$ such that
\[\skalp{x_k^\ast,x-x_k}\leq\frac 1k\norm{x-x_k}\ \forall x\in \M^{-1}(\yb)\cap \B(x_k,r_k).\]
By decreasing  the radii $r_k$ if necessary  we can assume $r_k/t_k\leq 1/k$. Further, by the supposed metric subregularity of $\M$ in direction $u$ with modulus $\kappa$ and by passing to a subsequence if necessary we have
\[\dist{x,\M^{-1}(\yb)}\leq (\kappa+\frac 1k) \dist{\yb,\M(x)}\ \forall x\in \B(x_k,r_k)\ \forall k.\]
Next fix $k$, consider $x\in \B(x_k,r_k/2)$ and let $\tilde x$ denote the projection of $x$ onto $\M^{-1}(\yb)$. Then $\norm{\tilde x-x_k}\leq \norm{x-\tilde x}+\norm{x-x_k}\leq 2\norm{x-x_k}\leq r_k$ and $\norm{x-\tilde x} = \dist{x,\mathcal{M}^{-1}(\bar{y})}\leq (\kappa+\frac 1k) \dist{\yb,\M(x)}$ and therefore
\begin{eqnarray*}\skalp{x_k^\ast,x-x_k}&=&\skalp{x_k^\ast,\tilde x-x_k}+\skalp{x_k^\ast,x-\tilde x}\leq \frac 1k\norm{\tilde x-x_k}+\norm{x_k^\ast}\norm{x-\tilde x}\\
&\leq&
\frac 2k\norm{x-x_k}+(\kappa+\frac 1k)\norm{x_k^\ast}\inf_{y\in \M(x)}\norm{y-\yb}.
\end{eqnarray*}
Since $x\in \B(x_k,r_k/2)$ was arbitrary, we conclude
\[(\kappa+\frac 1k)\norm{x_k^\ast}\norm{y-\yb}-\skalp{x_k^\ast,x-x_k}+\frac 2k\norm{x-x_k}\geq 0\ \forall x\in \B(x_k,r_k/2)\ \forall (x,y)\in\Gr M.\]
Taking into account that $x_k\in\M^{-1}(\yb)$ and therefore $\yb\in \M(x_k)$, we see that $(x_k,\yb)$ is a local minimizer  for the problem
\[\min_{(x,y)\in\Gr\M} (\kappa+\frac 1k)\norm{x_k^\ast}\norm{y-\yb}-\skalp{x_k^\ast,x-x_k}+\frac 2k\norm{x-x_k}.\]
The respective optimality conditions \cite[Theorem 8.15]{RoWe98} imply the existence of some $y_k^\ast\in\B_{\R^d}$ and some $\eta_k^\ast\in\B_{\R^s}$ such that
\[0\in\Big(-x_k^\ast+\frac 2k\eta_k^\ast,(\kappa+\frac 1k)\norm{x_k^\ast}y_k^\ast\Big)+N_{\Gr \M}(x_k,\yb)\]
 and by the definition of the limiting normal cone to $\Gr\M$ at $(x_{k}, \bar{y})$ we can find elements
$(\tilde x_k,\tilde y_k)\in\Gr\M$ and $(\tilde x_k^\ast,\tilde y_k^\ast)\in\widehat N_{\Gr \M}(\tilde x_k,\tilde y_k)$ such that
\[\norm{(\tilde x_k,\tilde y_k)-(x_k,\yb)}\leq \frac{t_k}k \quad\mbox{and}\quad  \norm{(-x_k^\ast+\frac 2k\eta_k^\ast,(\kappa+\frac 1k)\norm{x_k^\ast}y_k^\ast)+(\tilde x_k^\ast,\tilde y_k^\ast)}\leq \frac 1k.\]
We infer  $\norm{\tilde y_k^\ast}\leq (\kappa+\frac 1k)\norm{x_k^\ast}+\frac 1k$ and therefore, by passing to a subsequence if necessary, we can assume that $\tilde y_k^\ast$ converges to some $y^\ast\in\R^d$. Since $\lim_{k\to\infty}\tilde x_k^\ast=\lim_{k\to\infty}x_k^\ast= x^\ast$ and
\[\lim_{k\to\infty}\frac{(\tilde x_k,\tilde y_k)-(\xb,\yb)}{t_k}=\lim_{k\to\infty}\left(\frac{(\tilde x_k,\tilde y_k)-(x_k,\yb)}{t_k}+\lim_{k\to\infty}\frac{( x_k,\yb)-(\xb,\yb)}{t_k}\right)=(u,0),\]
we conclude $(x^\ast,y^\ast)\in N_{\Gr \M}((\xb,\yb);(u,0))$ and $\norm{y^\ast}\leq \kappa\norm{x^\ast}$ and the first inclusion in \eqref{EqInclDirNormal} is shown. The second one is straightforward.
\end{proof}

By combining \eqref{EqTangDir} and Lemma \ref{LemBasicPropDirMetr}(iii) we obtain the following corollary.
\begin{corollary}\label{CorTanCone}
  Assume that the multifunction $\M:\R^s\rightrightarrows \R^d$ is metrically subregular at $(\xb,\yb)\in\Gr \M$. Then
  \[T_{\M^{-1}(\yb)}(\xb)=\{u\mv 0\in D\M(\xb,\yb)(u)\}.\]
\end{corollary}

By combining the Mordukhovich criterion for the Aubin property of $S$ with the definition of the directional limiting coderivative and by invoking Proposition \ref{PropDirNormalCone} and Corollary \ref{CorTanCone}, we arrive at the next statement.

\begin{proposition}
  \label{PropAubinImplMultFunc}Assume that the condition
  \begin{equation}\label{EqAubPropImplMultiFunc1}(q^\ast,0)\in \{(q,u)\mv 0\in DM(\pb,\xb,0)(q,u)\}^\circ\ \Rightarrow\ q^\ast=0\end{equation}
  holds, assume that $M$ is metrically subregular at $(\pb,\xb,0)$ and assume that there do not exist vectors $(0,0)\not=(q,u)\in\R^l\times \R^n$, $(q^\ast,v^\ast)\in\R^l\times\R^m$ with $q^\ast\not=0$ such that
\begin{equation}
  \label{EqAubPropImplMultiFunc2}
  0\in DM(\pb,\xb,0)(q,u),\quad (q^\ast,0)\in  D^\ast M((\pb,\xb,0);(q,u,0))(v^\ast).
\end{equation}
Then $S$ has the Aubin property around $(\pb,\xb)$.
\end{proposition}
\begin{proof}By contraposition. Assume on the contrary that $S$ does not have the Aubin property around $(\pb,\xb)$. Then we can infer from the Mordukhovich criterion \cite{Mo92},\cite[Theorem 9.40]{RoWe98} that $0\not=q^\ast\in D^\ast S(\pb,\xb)(0)$. By the definition of the limiting coderivative there are sequences $(p_k,x_k,q_k^\ast,u_k^\ast)\to(\pb,\xb, q^\ast,0)$ such that $(q_k^\ast,-u_k^\ast)\in \widehat N_{\Gr S}(p_k,x_k)$ for every $k$. Consider first the case that $(p_k,x_k)=(\pb,\xb)$ holds for infinitely many $k$. Then, by passing to a subsequence and using the fact that the regular normal cone $\widehat N_{\Gr S}(\pb,\xb)$ is closed as a polar cone of $T_{\Gr S}(\pb,\xb)$, we obtain $(q^\ast,0)\in \widehat N_{\Gr S}(\pb,\xb)=(T_{\Gr S}(\pb,\xb))^\circ$. Since $M$ is metrically subregular at $(\pb,\xb,0)$ and $\Gr S=M^{-1}(0)$, by Corollary  \ref{CorTanCone} we arrive at $(q^\ast,0)\in \{(q,u)\mv 0\in DM(\pb,\xb,0)(q,u)\}^\circ$ contradicting \eqref{EqAubPropImplMultiFunc1}. Hence, $(p_k,x_k)=(\pb,\xb)$ only holds for finitely many $k$ and so without loss of generality $(p_k,x_k)\not=(\pb,\xb)$ for all $k$. By putting $t_k:=\norm{(p_k-\pb,x_k-\xb)}$ and by passing to a subsequence if necessary we can assume that $(p_k-\pb,x_k-\xb)/t_k$ converges to some $(q,u)$ with $\norm{(q,u)}=1$ and we conclude $(q^\ast,0)\in N_{\Gr S}((\pb,\xb);(q,u))$. Hence $(q,u)\in T_{\Gr S}(\pb,\xb)$ and by Proposition \ref{PropDirNormalCone} we conclude $0\in DM(\pb,\xb,0)(q,u)$ and the existence of some $v^\ast$ such that $(q^\ast,0)\in D^\ast M((\pb,\xb,0);(q,u,0))(v^\ast)$ contradicting \eqref{EqAubPropImplMultiFunc2}.
\end{proof}

We are now in the position to state the main result of this section:
\begin{theorem}
  \label{ThAubinImplMultFunc}Assume that
  \begin{equation}\label{EqAubPropImplMultiFunc1a}\{u\mv 0\in DM(\pb,\xb,0)(q,u)\}\not=\emptyset\quad \mbox{for all }q\in\R^l
  \end{equation}
   holds, assume that $M$ is metrically subregular at $(\pb,\xb,0)$ and  for every $(0,0)\not=(q,u)\in\R^l\times\R^n$ verifying $0\in DM(\pb,\xb,0)(q,u)$ the condition
\begin{equation}
  \label{EqAubPropImplMultiFunc2a}
   (q^\ast,0)\in  D^\ast M((\pb,\xb,0);(q,u,0))(v^\ast)\ \Longrightarrow\ q^\ast=0
\end{equation}
holds. Then $S$ has the Aubin property around $(\pb,\xb)$ and
  \begin{equation}\label{EqGraphDerivF} DS(\pb,\xb)(q)=\{u\mv 0\in DM(\pb,\xb,0)(q,u)\},\; q\in\R^l.
  \end{equation}
\end{theorem}
\begin{proof}
  In view of Proposition \ref{PropAubinImplMultFunc} and Corollary \ref{CorTanCone} we only have to show that condition \eqref{EqAubPropImplMultiFunc1a} implies condition \eqref{EqAubPropImplMultiFunc1}. In fact, let $(q^\ast,0)\in \{(q,u)\mv 0\in DM(\pb,\xb,0)(q,u)\}^\circ$. Then, for every $q\in\R^l$ we can find some $u_q\in
  \{u\mv 0\in DM(\pb,\xb,0)(q,u)\}$ and therefore $\skalp{q^\ast,q}+\skalp{0,u_q}=\skalp{q^\ast,q}\leq 0$ implying $q^\ast=0$.
\end{proof}

If we replace the requirement that $M$ is metrically subregular at $(\pb,\xb,0)$ by FOSCMS \eqref{EqEquivCrDirLimCoDer}, we obtain the next corollary.

\begin{corollary}\label{CorAubPropImplMultFunc}
  Assume that \eqref{EqAubPropImplMultiFunc1a} holds and assume that for every $(0,0)\not=(q,u)\in\R^l\times\R^n$ verifying $0\in DM(\pb,\xb,0)(q,u)$  the condition
  \begin{equation}
  \label{EqAubPropImplMultiFunc2b}
   (q^\ast,0)\in  D^\ast M((\pb,\xb,0);(q,u,0))(v^\ast)\ \Longrightarrow\ q^\ast=0, v^\ast=0
  \end{equation}
   is fulfilled. Then $S$ has the Aubin property around $(\pb,\xb)$ and \eqref{EqGraphDerivF} holds.
\end{corollary}

 We now show that condition \eqref{EqAubPropImplMultiFunc1a} is also necessary in order that the mapping $S$ has the Aubin property.
\begin{proposition}If $S$ has the Aubin property around $(\pb,\xb)$ then \eqref{EqAubPropImplMultiFunc1a} is fulfilled. 
\end{proposition}
\begin{proof}If $S$ has the Aubin property around $(\pb,\xb)$, then $\{\xb\}\subset S(p)+L\norm{p-\pb}\B_{\R^n}$ holds for all $p$ in some neighborhood of $\pb$. Consider now any direction $q\in\R^l$ and any sequence $t_k\searrow 0$. Then for every $k$ sufficiently large we can find $x_k\in S(\pb+t_kq)$ such that $\norm{x_k-\xb}\leq Lt_k\norm{q}$. By passing to a subsequence we can assume that  the sequence $u_k:=(x_k-\xb)/t_k$ converges to some $u$. Since $0\in M(\pb+t_kq_k,x_k)=M(\pb+t_kq_k,\xb+t_ku_k)$, the inclusion $0\in DM(\pb,\xb,0)(q,u)$ follows and thus $\{u\mv 0\in DM(\pb,\xb,0)(q,u)\}\not=\emptyset$. Because $q$ was chosen arbitrarily, relation \eqref{EqAubPropImplMultiFunc1a} follows.\end{proof}

Now let us compare the criteria of Corollary \ref{CorAubPropImplMultFunc} with the criterion of \cite[Corollary 4.60]{Mo06a}. By taking $f\equiv 0$ and $Q=M$ in \cite[Corollary 4.60]{Mo06a} we obtain that the condition
\begin{equation}
  \label{EqAubPropImplMultiFuncMord}
   (q^\ast,0)\in  D^\ast M(\pb,\xb,0)(v^\ast)\ \Longrightarrow\ q^\ast=0,\ v^\ast=0
\end{equation}
is sufficient for the Aubin property of $S$ around $(\pb,\xb)$. So, instead of the standard coderivative of $M$ used in \eqref{EqAubPropImplMultiFuncMord}, we use  in condition \eqref{EqAubPropImplMultiFunc2b} the directional coderivative of $M$ in certain directions, which is by definition not larger (typically smaller) than the standard coderivative. This indicates that in this way we arrive at substantially less restrictive sufficient conditions ensuring the Aubin property of $S$. By  Example \ref{ExAubin} below, we will strikingly illustrate  that the conditions of Corollary \ref{CorAubPropImplMultFunc} are indeed weaker than \eqref{EqAubPropImplMultiFuncMord}.

Before we present this example, we work out the preceding theory for the case of a class of  variational systems, where
\begin{equation}\label{EqVarSystem}
  M(p,x)=G(p,x)+Q(x)
\end{equation}
with $G:\R^l\times\R^n\to\R^m$ continuously differentiable and $Q:\R^n\rightrightarrows\R^m$ being a closed-graph multifunction. It is well-known that in this case, cf. \cite[Proposition 4A.2]{DoRo14}, at a fixed triple $(\pb,\xb,0)\in\Gr M$ one has
\[DM(\pb,\xb,0)(q,u)=\nabla_p G(\pb,\xb)q+\nabla_xG(\pb,\xb)u+DQ(\xb,\yba)(u),\]
where $\yba:=-G(\pb,\xb)$.

Condition \eqref{EqAubPropImplMultiFunc1a} thus amounts to the requirement that the {\em generalized equation} (GE)
\begin{equation}
  \label{EqLinVarSyst}0\in \nabla_pG(\pb,\xb)q+\nabla_xG(\pb,\xb)u+DQ(\xb,\yba)(u)
\end{equation}
in variable $u$ possesses a solution for all $q\in\R^l$. Further, by virtue of \eqref{EqSumMF_Smooth_LimCoDeriv}, condition \eqref{EqAubPropImplMultiFunc2a} amounts to the implication
\begin{equation}
  \label{EqAubPropImplMultiFunc2aVarSyst}
 \left. \begin{array}{l}
    q^\ast=(\nabla_p G(\pb,\xb))^Tv^\ast\\
    0\in(\nabla_xG(\pb,\xb))^Tv^\ast+D^\ast Q((\xb,\yba);(u,-\nabla_pG(\pb,\xb)q-\nabla_xG(\pb,\xb)u))(v^\ast)
  \end{array}\right\}\Longrightarrow q^\ast=0
\end{equation}
and condition \eqref{EqAubPropImplMultiFunc2b} amounts to a strengthened variant of \eqref{EqAubPropImplMultiFunc2aVarSyst}, where on the right-hand side one has $q^\ast=0$, $u^\ast=0$. In contrast to the criterion from \cite[Corollary 4.61]{Mo06a} this means that, instead of the solutions $v^\ast$ to the standard adjoint GE
\begin{equation}
  \label{EqStandAdjGE} 0\in(\nabla_x G(\pb,\xb))^Tv^\ast+D^\ast Q(\xb,\yba)(v^\ast)
\end{equation}
with the standard limiting coderivative of $Q(\cdot)$, we have now to consider the respective {\em directional adjoint GE} for directions $(q,u)$ solving \eqref{EqLinVarSyst}. By the definition, the respective set of solutions is not larger (typically much smaller) than the set of solutions to \eqref{EqStandAdjGE}.

In the following example we illustrate the efficiency of our technique for the special case when $Q(x)=N_\Gamma(x)$ with $\Gamma\subset\R^n$ being a convex polyhedron. Then, by virtue of \eqref{EqPolyTanCone}, condition \eqref{EqLinVarSyst} amounts to
\begin{equation}
  \label{EqLinVarInequ} 0\in \nabla_pG(\pb,\xb)q+\nabla_xG(\pb,\xb)u+N_{K_\Gamma(\xb,\yba)}(u).
\end{equation}
\begin{example}
  \label{ExAubin} Consider the solution map $S:\R\rightrightarrows\R^2$ of the GE
  \begin{equation}
    \label{EqEx5} 0\in M(p,x)=\left(\begin{array}{c}
      x_1-p\\-x_2+x_2^2
    \end{array}\right)+N_\Gamma(x)
  \end{equation}
  with $\Gamma=\{x\in\R^2\mv \frac 12 x_1\leq x_2\leq -\frac 12 x_1\}$ at $(\pb,\xb)=(0,(0,0))\in \Gr S$. We will now demonstrate that by means of Corollary \ref{CorAubPropImplMultFunc} we can verify the Aubin property of $S$ around $(\pb,\xb)$ whereas  \cite[Corollary 4.60, Corollary 4.61]{Mo06a} are not applicable.
  
  We have $\yba=(0,0)$, $K_\Gamma(\xb,\yba)=\{u\in\R^2\mv \frac 12 u_1\leq u_2\leq -\frac 12u_1\}$ and therefore, concerning the directions $(q,u)$ verifying \eqref{EqLinVarInequ}, one has to consider the following four situations:
  \begin{enumerate}
    \item[(i)]$q\leq 0$, $u_1=q$, $u_2=0$;
    \item[(ii)]$q\leq 0$, $u_1=\frac 43 q$, $u_2=-\frac 23 q$;
    \item[(iii)]$q\leq 0$, $u_1=\frac 43 q$, $u_2=\frac 23 q$;
    \item[(iv)]$q\geq 0$, $u_1=u_2=0$.
  \end{enumerate}
  We observe that \eqref{EqAubPropImplMultiFunc1a} is fulfilled.
  Since we are only interested in nonzero directions $(q,u)$, for our further analysis  we can restrict to the case $q\not=0$.

The faces of the critical cone are exactly the cones ${\cal F}_1:=\{(0,0)\}$, ${\cal F}_2:=\R^+(-1,\frac 12)$, ${\cal F}_3:=\R^+(-1,-\frac 12)$ and the critical cone ${\cal F}_4:=K_\Gamma(\xb,\yba)$ itself.

  In the case (i) one has
  \[-\nabla_pG(\pb,\xb)q-\nabla_xG(\pb,\xb)u=\left(\begin{array}
    {c}0\\0
  \end{array}\right)\]
   and, by virtue of Theorem \ref{ThPolyDirLimNormCone},
  \[ N_{\Gr N_\Gamma}((\xb,\yba);((q,0),(0,0)))=({\cal F}_4-{\cal F}_4)^\circ\times ({\cal F}_4-{\cal F}_4)=\{(0,0)\}\times\R^2,\]
  since the only face of $K_\Gamma(\xb,\yba)$ containing $(q,0)$ with $q<0$ is the critical cone itself. Thus
  \[ D^\ast N_\Gamma((\xb,\yba);((q,0),(0,0)))(v^\ast)=\{(0,0)\}\]
  and the directional adjoint GE attains the form
  \begin{equation}
    \label{EqEx5GE1}\left(\begin{array}
    {c}0\\0
  \end{array}\right)=\left(\begin{array}{c}v_1^\ast\\-v_2^\ast\end{array}\right).
  \end{equation}

  In the case (ii)
  \[-\nabla_pG(\pb,\xb)q-\nabla_xG(\pb,\xb)u=\left(\begin{array}
    {c}-\frac 13 q\\-\frac 23 q
  \end{array}\right),\]
  \[ N_{\Gr N_\Gamma}((\xb,\yba);((\frac 43 q,-\frac 23q),(-\frac 13 q,-\frac 23 q)))=({\cal F}_2-{\cal F}_2)^\circ\times ({\cal F}_2-{\cal F}_2),\]
 \[ D^\ast N_\Gamma((\xb,\yba);((\frac 43 q,-\frac 23q),(-\frac 13 q,-\frac 23 q)))(v^\ast)=
  \begin{cases}{\cal K}_1^\circ&\mbox{if $-v^\ast\in {\cal K}_1$,}\\
  \emptyset&\mbox{otherwise,}\end{cases}\]
  with ${\cal K}_1:={\cal F}_2-{\cal F}_2=\R(-1,\frac 12)$
  and the directional adjoint GE attains the form
  \begin{equation}
    \label{EqEx5GE2}\left(\begin{array}
    {c}0\\0
  \end{array}\right)\in\left(\begin{array}{c}v_1^\ast\\-v_2^\ast\end{array}\right)+\R\left(\begin{array}{c}\frac 12\\ 1\end{array}\right),\  -v^\ast\in {\cal K}_1,
  \end{equation}
  which has the only solution $v^\ast=0$.

   Similarly to the second case, in the case (iii) the directional adjoint GE attains the form
   \begin{equation}
    \label{EqEx5GE3}\left(\begin{array}
    {c}0\\0
  \end{array}\right)\in\left(\begin{array}{c}v_1^\ast\\-v_2^\ast\end{array}\right)+\R\left(\begin{array}{c}\frac 12\\ -1\end{array}\right),\  -v^\ast\in {\cal K}_2:={\cal F}_3-{\cal F}_3=\R\left(\begin{array}{c}-1\\-\frac 12\end{array}\right)
  \end{equation}
  and again the unique solution is $v^\ast=0$.

  Finally, in the case (iv),
  \[-\nabla_pG(\pb,\xb)q-\nabla_xG(\pb,\xb)u=\left(\begin{array}
    {c} q\\0  \end{array}\right),\]
  \[ N_{\Gr N_\Gamma}((\xb,\yba);((0,0),(q,0)))=({\cal F}_1-{\cal F}_1)^\circ\times({\cal F}_1-{\cal F}_1)=\R^2\times\{(0,0)\},\]
  \[ D^\ast N_\Gamma((\xb,\yba);((0,0),(q,0)))(v^\ast)=
  \begin{cases}\R^2&\mbox{if $v^\ast=(0,0)$,}\\
  \emptyset&\mbox{otherwise,}\end{cases}\]
  and the directional adjoint GE attains the form
  \begin{equation}
    \label{EqEx5GE4}\left(\begin{array}
    {c}0\\0
  \end{array}\right)\in\left(\begin{array}{c}v_1^\ast\\-v_2^\ast\end{array}\right)+\R^2,\  v^\ast=(0,0).
  \end{equation}
  In this way we have analyzed all ''suspicious'' pairs of nonzero directions $(q,u)$ given by \eqref{EqLinVarInequ} and concluded that all GEs \eqref{EqEx5GE1}-\eqref{EqEx5GE4} possess only the trivial solution $v^\ast=(0,0)$. Since $q^\ast=-v_1^\ast$, Corollary \ref{CorAubPropImplMultFunc} implies that the solution map of GE \eqref{EqEx5} indeed has the Aubin property around $(\pb,\xb)$.

    Now let us analyze the standard GE \eqref{EqStandAdjGE}, which reads as
  \begin{equation}\label{EqEx5StandardGE}\left(\begin{array}
    {c}0\\0
  \end{array}\right)\in\left(\begin{array}{c}v_1^\ast\\-v_2^\ast\end{array}\right)+D^\ast N_\Gamma(\xb,\yba)(v^\ast)
  \end{equation}
  for our example. Using the representation of  the limiting normal cone $N_{\Gr  N_\Gamma}$ at $(\xb,\yba)$ as stated in Section \ref{SecPrelim}, we obtain
  \[N_{\Gr \widehat N_\Gamma}(\xb,\yba)=\bigcup_{i=1}^9({\cal K}_i^\circ\times {\cal K}_i)\]
  with ${\cal K}_3={\cal F}_4-{\cal F}_1=K_\Gamma(\xb,\yba)$, ${\cal K}_4={\cal F}_4-{\cal F}_2=\{v\in\R^2\mv \frac 12 v_1+v_2\leq 0\}$, ${\cal K}_5={\cal F}_4-{\cal F}_3=\{v\in\R^2\mv \frac 12 v_1-v_2\leq 0\}$, ${\cal K}_6={\cal F}_4-{\cal F}_4=\R^2$,  ${\cal K}_7={\cal F}_2-{\cal F}_1=\R^+(-1,\frac 12)$, ${\cal K}_8={\cal F}_3-{\cal F}_1=\R^+(-1,-\frac 12)$ and ${\cal K}_9={\cal F}_1-{\cal F}_1=\{(0,0)\}$.\\
  We see that for $v^\ast=(-1,2)$ we have $-v^\ast\in{\cal K}_4$ and $-(v_1^\ast,-v_2^\ast)=(1,2)\in {\cal K}_4^\circ\subset D^\ast N_\Gamma(\xb,\yba)(v^\ast)$, verifying that $v^\ast$ is a nontrivial solution of the GE \eqref{EqEx5StandardGE}. Another nontrivial solution of \eqref{EqEx5StandardGE} is provided by $v^\ast=(-1,-2)\in-{\cal K}_5$. This implies that we cannot apply \cite[Corollary 4.60, Corollary 4.61]{Mo06a} to detect the Aubin property of the solution map $S$.\hfill$\triangle$
\end{example}

\section{Conclusion}

In both main sections of the paper (i.e., Sections 3 and 4) we   use as basic tool the directional limiting coderivatives. The purpose for their usage, however, is different. Whereas in Section 3 they are employed in verifying the calmness of $M_{\pb }$ and in this role they could possibly be substituted by another calmness criterion, in Section 4 they help us to capture the behavior of $M$ along relevant directions and in this role they cannot be substituted by any of the currently available generalized derivatives. This ability of directional limiting coderivatives could possibly be utilized also in analysis of other stability properties of $S$ (than the calmness and the  Aubin property). In particular, under the assumptions of Theorem \ref{ThAubinImplMultFunc} the mapping $S$ has a single-valued Lipschitz localization around $(\bar{p},\bar{x})$ whenever we ensure the single-valuedness of $S$ close to $(\bar{p},\bar{x})$. This may be done, e.g., by standard monotonicity assumptions imposed on $M_{p}(\cdot)$ but we believe that a suitable additional condition could be formulated directly in terms of graphical derivatives and directional limiting coderivatives of $M$ at $(\bar{p},\bar{x}, 0)$. This question we plan to tackle in our future research. The application potential of the directional limiting coderivative is further increased by the formula developed in Theorem \ref{ThPolyDirLimNormCone},  which  enables us an efficient computation of this object in the case of polyhedral constraints.

\section*{Acknowledgements}
The authors would like to express their gratitude to both reviewers for their numerous important suggestions. The research of the first author was supported by the Austrian Science Fund (FWF) under grant P26132-N25. The research of the second author was supported by the Grant Agency of the Czech Republic, project 15-00735S and the Australian Research Council, project  DP160100854.

\end{document}